\documentclass[aos]{imsart}
\usepackage{graphicx}
\usepackage{amsfonts}
\usepackage{amsmath}
\usepackage{url}
\usepackage{amsthm}
\usepackage{amssymb}
\newtheorem {Proposition}{Proposition}[section]
\newtheorem {Lemma}[Proposition] {Lemma}
\newtheorem {Theorem}[Proposition]{Theorem}
\newtheorem {Corollary}[Proposition]{Corollary}
\newtheorem {Remark}[Proposition]{Remark}

\def\N{\mathbb{N}}

\def\R{\mathbb{R}}
 
\def\E{\mathbb{E}}

\usepackage{enumerate}
\usepackage[sort,compress]{natbib}
\usepackage{bbm}
\usepackage[T1]{fontenc}    
\usepackage{lmodern}        
\usepackage{hyperref}       
\usepackage{url}            
\usepackage{booktabs}       
\usepackage{amsfonts}       

\usepackage{nicefrac}       
\usepackage{microtype}      
\usepackage{lipsum}
\usepackage{graphicx}
\usepackage{color}
\graphicspath{ {./images/} }

\begin{document}
\begin{frontmatter}
\title{An improved central limit theorem and fast convergence rates for entropic transportation costs}

\runtitle{CLT  and fast convergence rates for entropic transportation costs}

\begin{aug}

\author[A]{\fnms{Eustasio} \snm{del Barrio}\ead[label=e1]{tasio@eio.uva.es}},
\author[B]{\fnms{Alberto} \snm{Gonz\'alez-Sanz}\ead[ label=e2]{alberto.gonzalezsanz@math.univ-toulouse.fr}},
\author[C]{\fnms{Jean-Michel} \snm{Loubes}\ead[label=e3]{ loubes@math.univ-toulouse.fr}}
\and
\author[D]{\fnms{Jonathan} \snm{Niles-Weed}\ead[label=e4]{ jnw@cims.nyu.edu}}
\address[A]{IMUVA, Universidad de Valladolid, Spain.
\printead{e1}}

\address[B]{IMT, Universit\'e de Toulouse,
France. \printead{e2}}
\address[C]{IMT, Universit\'e de Toulouse,
France. \printead{e3}}
\address[D]{Courant Courant Institute for Mathematical Sciences and Center for Data Science, New York University 
\printead{e4}}
\end{aug}

\begin{abstract}
We prove a central limit theorem for the entropic transportation cost between subgaussian probability measures, centered at the population cost. This is the first result which allows for asymptotically valid inference for entropic optimal transport between measures which are not necessarily discrete. In the compactly supported case, we complement these results with new, faster, convergence rates for the expected entropic transportation cost between empirical measures. Our proof is based on strengthening convergence results for dual solutions to the entropic optimal transport problem.
\end{abstract}
\end{frontmatter}

\section{Introduction.}
Optimal transport has emerged as a leading methodology in many areas of data science, machine learning, and statistics~\cite{RigWee18a,HutRig21,20-AOS1996,GhoSen19,deb2019multivariate,HalMorSeg21,MorSze20,AbaImb06, Tod10, MorHar06,BlaKanMur19,KuhEsfNgu19,FlaCutCou18,GaoKle16,GorDelFab19, Barrio2019ACL,BlaYeoFre20, ChiPac21, ChiJiaSte20,LiZhaLuo20, CouFlaTui17, RedHabSeb17,GraJouBer19,AlvJaa18,delara2021transportbased}, with applications in fields ranging from high-energy physics~\cite{RomCasMil20, KomMetTha19}  to computational biology~\cite{schiebinger2019optimal,dai2018autoencoder,shi2020rate,shi2021center,MR4399094}.
Central to its recent success in practice is the paradigm of entropic regularization, popularized by \cite{Cut13}, which leads to a highly efficient parallelizable algorithm suitable for large-scale data analysis~\cite{peyreCuturi}.
This regularization is defined by augmenting the standard optimal transportation problem by a penalization term based on relative entropy, defined between two probability measures $\alpha$ and $\beta$ as $H(\alpha | \beta) = \int \log(\frac{d\alpha}{d\beta}(x))d\alpha(x)$ if $\alpha$ is absolutely continuous with respect to $\beta$, $\alpha\ll \beta$, and $+\infty$ otherwise.
Given $P, Q \in \mathcal{P}(\R^d)$ and $\epsilon > 0$, the resulting problem reads
\begin{equation}\label{kanto_entrop}
	S_{\epsilon}(P,Q)=\min_{\pi\in \Pi(P,Q)} \int_{\R^d\times \R^d} {\textstyle \frac{1}{2}}\|x-y\|^2 d\pi(x,y)+\epsilon H(\pi | P\times Q),
\end{equation}
where $\Pi(P, Q)$ denotes the set of couplings between $P$ and $Q$.

Alongside its computational virtues, entropic regularization also brings substantial statistical benefits: entropically regularized transportation costs enjoy faster convergence rates than their unregularized counterparts, even in high dimensions, making them useful for estimation tasks~\cite{genevay2019,Weed19,chizat2020faster,PooNil21}.
Moreover, entropic regularization seems well suited to problems involving data corrupted with Gaussian noise~\cite{RigWee18}.
Together, this body of results suggests the strengths of entropic optimal transport as an applied and theoretical statistical tool.

Obtaining limit theorems for the unregularized transportation costs is a longstanding question in probability theory and statistics.
(See the recent work~\cite{HunKlaSta22} for references and an account of the history of this problem.)
Under relatively stringent assumptions on the measures, it is known that the empirical unregularized transport costs possesses asymptotically Gaussian fluctuations around its expectation~\cite{delbarrio2019,delbarrio2021central}; stronger results can be obtained when one or both of the measures is discrete~\cite{delbarrio2021centraldisc,Sommerfeld2018,tameling2019}, when the measures are smooth~\cite{ManBalNil21}, and in one dimension~\cite{BarGinMat99,BarGinUtz05}.

The strict convexity and differentiability of the regularized optimal transportation problem makes it possible to prove significantly more general results.
A central limit theorem for entropy regularized transportation costs, centered at the expectation of the empirical cost, was first obtained by~\cite{Weed19} (see \eqref{mena_clt} below).
Generalizations and extensions for discrete measures have been proved by~\cite{KlaTamMun20,BigotCLT}.
A growing body of work investigates the properties of the entropy regularized optimal transport problem from the perspective of probability and analysis, including its asymptotic properties as $\epsilon \to 0$~\cite{NutWie21,EckNut21,GhoNutBer21,NutWie22,AltNilStr22,Ber20,CheGalHal17}, opening the door to further statistical applications of entropy regularised transport.

A crucial question in statistical applications of entropic optimal transport costs is the construction of asymptotic confidence intervals, to permit asymptotically valid inference.
The most general results known in this direction are due to~\cite{Weed19}, who showed that if $P$ and $Q$ are subgaussian probabilities on $\mathbb{R}^d$, then
\begin{equation}\label{mena_clt}
    \sqrt{n}(S_{\epsilon}(P_n,Q)-\E S_{\epsilon}(P_n,Q))\xrightarrow{w} N(0,\operatorname{Var}_P(f^*_\epsilon)),
\end{equation}
with $S_\epsilon(\cdot, \cdot)$ as in \eqref{kanto_entrop}.
(See Section~\ref{sec:prelim} for further background and definitions.)
A limitation of this result in practical inference problems is the centering at $\E S_\epsilon(P_n, Q)$ rather than at the population quantity $S_\epsilon(P, Q)$.
This result parallels known results for the unregularized transport cost: \cite{delbarrio2019,delbarrio2021central} show that, under suitable technical conditions on $P$ and $Q$, there exists $\sigma \geq 0$ such that
\begin{equation}
	\sqrt{n}(W_p^p(P_n, Q) - \E W_p^p(P_n, Q)) \xrightarrow{w} N(0, \sigma^2)\,,
\end{equation}
where $W_p$ denotes the unregularized $p$-transportation cost, $p > 1$. See also \cite{gonzalezdelgado2021twosample} for its generalization to the flat torus.
In this case, it is know that the centering at $\E W_p^p(P_n, Q)$ is unavoidable, and that it is not possible in general to replace $\E W_p^p(P_n, Q)$ by $W_p^p(P, Q)$, in view of the fact that known lower bounds on convergence rates of the Wasserstein distance imply that $\sqrt{n} (W_p^p(P_n, Q) - W_p^p(P, Q))$ is typically not stochastically bounded when $d > 2p$.

However, this limitation does \emph{not} apply to the entropically regularized transport costs.
Indeed, the results of \cite{Weed19} imply that
\begin{equation}\label{mean_deviation}
	|S_\epsilon(P, Q) - \E  S_{\epsilon}(P_n, Q)| \leq C_{P, Q} n^{-1/2}
\end{equation}
for a positive constant $C_{P, Q}$ depending on the two measures.
As a consequence, prior work does not rule out the possibility that $\sqrt{n}(S_\epsilon(P_n, Q) - S_\epsilon(P, Q))$ enjoys a central limit theorem, but nor does it provide a proof that such a theorem holds.

In this paper, we close this gap.
We show a central limit theorem of the form
\begin{equation}\label{main_eq}
	 \sqrt{n}(S_{\epsilon}(P_n,Q) - S_{\epsilon}(P,Q))\xrightarrow{w} N(0,\operatorname{Var}_P(f^*_\epsilon)),
\end{equation}
valid for any subgaussian probabilities $P$ and $Q$ in any dimension.
Prior to our work, such a bound was known only when $P$ and $Q$ were supported on a finite or countable set~\cite{KlaTamMun20,BigotCLT}. Also \cite{Harchaoui2020AsymptoticsOE} provides central limit theorems  for a different entropic regularization, where the solution is explicit.
Our results represent a significant generalization of these results, and imply that, under sufficiently strong moment conditions, asymptitcally valid inference is always possible for the entropic transportation cost.

Our proof of \eqref{main_eq} is based on an important strengthening of~\eqref{mean_deviation}.
Specifically, we show that, for subgaussian probability measures,
\begin{equation}\label{improved}
	|\E  S_\epsilon(P_n, Q) - S_{\epsilon}(P, Q)| = o(n^{-1/2}) \quad \text{as $n \to \infty$.}
\end{equation}
Combining this result with \eqref{mena_clt} yields \eqref{main_eq}.

When $P$ and $Q$ are supported on a bounded set $\Omega$, we are able to obtain substantially more precise results, which are of independent interest.
Our techniques imply that for compactly supported $P$ and $Q$,
\begin{equation*}
	|\E  S_\epsilon(P_n, Q) - S_{\epsilon}(P, Q)| \leq C_{P, Q} n^{-1}\,.
\end{equation*}
(See Remark~\ref{fast_rate}.)
This result implies that the bias of $S_\epsilon(P_n, Q)$ decays at the fast $n^{-1}$ rate, thereby recovering the rate typically obtained for \emph{parametric} estimation problems.
Our proof also yields new sample complexity results for the Sinkhorn divergence, defined as $D_\epsilon(P, Q) = S_\epsilon(P, Q) - \frac 12 (S_\epsilon(P, P) + S_\epsilon(Q, Q))$.
For probability measures on compact sets, convergence in Sinkhorn divergence is equivalent to weak convergence~\cite{Feydy2019InterpolatingBO}, implying that $D_\epsilon(P_n, P) \to 0$ a.s.
In Theorem~\ref{divergence_thm}, we show the quantitative bound
\begin{equation}
	\E  D_\epsilon(P_n, P) \leq C_{P} n^{-1},
\end{equation}
valid for all compactly supported $P$.
This convergence rate could have been anticipated from known distributional limits for Sinkhorn divergences between finitely supported measures~\cite{BigotCLT,KlaTamMun20}, but was unknown prior to our work.

In the bounded case, these results are all derived as corollaries of new convergence results for the  \emph{optimal dual potentials} in the entropic transport problem.
In Theorem \ref{Theorem_Potential}, we prove that, when $P$ and $Q$ are bounded, the entropic potentials converge fast in H\"older norm:
\begin{equation}\label{dual_converge}
\E \| g_n-g^*\|^2_{\mathcal{C}^s(\Omega)},\ \E \| f_n-f^*\|^2_{\mathcal{C}^s(\Omega)} \leq C_{P, Q} n^{-1}\,,
\end{equation}
where $s = [d/2]+1$.
We prove this result, as well as its two-sample analogue, in Section \ref{section3}.
To our knowledge, these bounds for the derivatives are new, even for finitely supported probability measures. Similar bounds for the potentials have been proved using completely different points of view, in \cite{luise2019sinkhorn}. It should be noted that the potentials' bound they provide is not enough to prove Theorem~\ref{divergence_thm}, at least with the technics exposed in this work.\\
When $P$ and $Q$ are not necessarily bounded but have subgaussian tails, we prove a non-quantitative analogue of \eqref{dual_converge}, showing that $f_n$ and $g_n$ converge to $f^*$ and $g^*$ almost surely in a suitably strong topology.
This result is a strengthening of a similar convergence result obtained by~\cite{Weed19}.

The remaining sections of this paper are organized as follows. Section~\ref{sec:prelim} provides some background results on entropic transportation costs. The central limit theorem \eqref{main_eq} and the faster rate \eqref{improved} are given in  Section~\ref{sec:TCL}. 
Section~\ref{section3} contains the announced results about the convergence rates of the potentials. The bounds for  Sinkhorn divergences are proved in Section~\ref{sec:divergences}. Finally we include a section with some numerical illustration of our limit theorems.

\section{Preliminaries on entropic transportation costs}\label{sec:prelim}
This selection collects several background results on the entropic transportation problem \eqref{kanto_entrop}.

We say that a distribution $\nu$ is the pushforward by a map $T$ of a distribution $\mu$, if $\nu=\mu\circ T^{-1}$. A simple computation shows that if $P^{\varepsilon}$ and $Q^{\varepsilon}$ denote the pushforwards of $P$ and $Q$ under the map $x\mapsto \varepsilon^{-\frac 1 2}x$ then $S_{\epsilon}(P,Q)=\varepsilon S_1(P^{\varepsilon},Q^{\varepsilon})$. Hence, we focus on the case $\varepsilon=1$ and write simply $S(P,Q)$ instead of $S_1(P,Q)$. 
The minimisation problem  \eqref{kanto_entrop} admits a dual formulation. In fact, if $\pi\in \Pi(P,Q)$ and $r=\frac{d\pi}{d(P\times Q)}$, then, for any $f\in L_1(P)$,
$g\in L_1(Q)$
\begin{eqnarray*}
	\lefteqn{\int [{\textstyle \frac{1}{2}}\|x-y\|^2 +\log r(x,y)] r(x,y)dP(x)dQ(y)\geq \int f(x) dP(x)+\int g(y)dQ(y)}\hspace*{6cm}\\
	&-&  \int  e^{{f(x)+g(y)-\frac 1 2\|x-y\|^2}}  dP(x)dQ(y)+1,
\end{eqnarray*}
with equality if and only if $r(x,y)=e^{f(x)+g(y)-\frac 1 2\|x-y\|^2}$ $P\times Q$-almost surely.
(This follows from the elementary fact that $s\log s\geq s-1$, $s>0$, with equality if and only if $s=1$).
This inequality implies the following version of weak duality:
\begin{eqnarray*}
	\lefteqn{S(P,Q)\geq \sup_{f\in L_1(P),g\in L_1(Q)} \Big\{\int_{\R^d} f(x) d P(x)+\int_{\R^d} g(y) d Q(y)}\hspace*{4cm}\\
	&-& \int_{\R^d\times \R^d}e^{{f(x)+g(y)- \frac{1}{2}\|x-y\|^2}} d P(x)dQ(y)+1\Big\}.
\end{eqnarray*}
It shows also that if $\frac{d\pi}{d(P\times Q)}=e^{{f(x)+g(y)-\frac 1 2\|x-y\|^2}}$ for some $f \in L_1(P_1)$ and $g \in L_1(Q)$, then $\pi$ is a minimizer for
the entropic transportation problem (indeed, by the strict convexity of $H$, it is the unique minimizer). The theory of entropic optimal transportation (see \cite{Csiszar1975,Nut21}) shows that the last inequality is, in fact, an equality, namely,
\begin{eqnarray}\nonumber
	\lefteqn{ S(P,Q)= \sup_{f\in L_1(P),g\in L_1(Q)} \Big\{\int_{\R^d} f(x) d P(x)+\int_{\R^d} g(y) d Q(y)}\hspace*{4cm}\\
	\label{dual_entrop}
	&-&\int_{\R^d\times \R^d}e^{{f(x)+g(y)- \frac{1}{2}\|x-y\|^2}} d P(x)dQ(y)+1\Big\}.
\end{eqnarray}
Maximizing pairs in \eqref{dual_entrop} are called optimal potentials. {These optimal potentials exist and satisfy some regularity conditions under integrability assumptions on $P$ and $Q$. 
	
	Following the framework in~\cite{Weed19}, we say that a probability $P$ is $\sigma^2$-subgaussian if $\E \big(e^{\frac{\|X\|^2}{2d\sigma^2}}\big)\leq 2$ when $X\sim P$.  When $P$ and $Q$ are subgaussian there exist optimal potentials}, denoted by $f^*,g^*$, satisfying the \emph{optimality conditions}, i.e. 
\begin{align}\nonumber
	&\int e^{{f^*(x)+g^*(y)- \frac{1}{2}\|x-y\|^2}}dQ(y)  =1, \ \ \text{for all $x\in \R^d $},\\
	\label{optimalli}
	&\int e^{{f^*(x)+g^*(y)- \frac{1}{2}\|x-y\|^2}}dP(x)  =1, \ \ \text{for all $y\in \R^d $},
\end{align}
see Proposition 6 in \cite{Weed19}. 
Moreover, the pair $(f^*, g^*)$ satisfying \eqref{optimalli} is unique up to constant shifts, and is uniquely specified by adopting the normalization convention
\begin{equation}\label{normalize}
	\int f^*(x) dP(x) = \int g^*(y) dQ(y).
\end{equation}
In what follows, we tacitly assume that \eqref{normalize} holds unless we explicitly specify an alternate convention.

The above considerations imply that the minimizer in the primal formulation is
$$d\pi^*=e^{{f^*(x)+g^*(y)- \frac{1}{2}\|x-y\|^2}}dQ(y)dP(x),$$ 
where $f^*$ and $g^*$ satisfy
\begin{align}\nonumber
	&f^*(x)=-\log\left(\int e^{{g^*(y)- \frac{1}{2}\|x-y\|^2}}dQ(y) \right),\\ \label{optimalli2} 
	&g^*(y)=-\log\left(\int e^{{f^*(y)- \frac{1}{2}\|x-y\|^2}}dP(y) \right).
\end{align}
Let $\alpha=(\alpha_1, \dots, \alpha_d)\in \N^{d}$ be a multi-index. If $P,Q\in \mathcal{P}(\R^d)$ are $\sigma^2$-subgaussian then (see
Proposition 1 in \cite{Weed19}), the optimal potential $f^*$ specified above is such that
\begin{equation}\label{bound1}
	|D^{\alpha}(f^*-{\textstyle\frac{1}{2}}\|\cdot \|^2)(x)|\leq    C_{k,d}   \begin{cases}
		1+\sigma^4 & \text{ if  } k=0\\
		\sigma^k(\sigma+\sigma^2)^k & \text{ otherwise},\\
	\end{cases}\qquad \text{if } \|x\|\leq \sqrt{d}\sigma,
\end{equation}
\begin{equation}\label{bound2}
	|D^{\alpha}(f^*-{\textstyle\frac{1}{2}}\|\cdot \|^2)(x)|\leq   C_{k,d}  \begin{cases}
		1+(1+\sigma^2)\| x\|^2 & \text{ if  } k=0\\
		\sigma^k(\sqrt{\sigma \|x\|}+\sigma \| x\|)^k & \text{ otherwise},\\
	\end{cases} \text{if } \|x\|\geq \sqrt{d}\sigma,
\end{equation}
and likewise for $g^*$, where in both cases $k:=|\alpha|$, and the constant $ C_{k,d} $ depends only on $d$ and $k$.

Throughout the paper, we will assume that $P,Q\in \mathcal{P}(\R^d)$ are $\sigma^2$-subgaussian probabilities and  $X_1, \dots, X_n$ and $Y_1, \dots, Y_m$ are independent samples of i.i.d r.v.'s with laws $P$ and $Q$, respectively. We will denote by $P_n$ and $Q_m$ the associated empirical measures.
We will require that the measures $P_n$ and $Q_n$ are also subgaussian, which is guaranteed by the following result, which summarizes Lemma 2 and Lemma 4 in \cite{Weed19}.
\begin{Lemma}\label{Lemma_weed}
	Let $X_1, \dots, X_n$ be i.i.d random variables with $\sigma^2$-subgaussian law $P\in \mathcal{P}(\R^d)$ and let $P_n$ be the associated empirical measure. Then, there exists a random variable $\tilde{\sigma}$, such that
	\begin{enumerate}[(i)]
		\item for every $n\in \N$, the probabilities $P$ and $P_n$ are uniformly $\tilde{\sigma^2}$-subgaussian almost surely,
		\item for any $k\in \N$, 
		we have $\E \left( \tilde{\sigma}^{2k}  \right)\leq  2k^k\sigma^{2k}$.
	\end{enumerate}
\end{Lemma}
\section{An improved central limit theorem for subgaussian probability measures}\label{sec:TCL}
This section shows that, for subgaussian probability measures, the expected empirical entropic transportation cost converges to its population counterpart with rate $o(n^{-1/2})$. This is an improvement over the bound \eqref{mean_deviation} derived in \cite{Weed19} and has, as a main consequence, a CLT for the empirical entropic transportation cost with the natural centering constants (see Theorem \ref{NaturalCLT} below), which, in turn, yields an asymptotically valid confidence interval for $S_\varepsilon(P,Q)$ regardless the dimension, $d$.

{Let $s$ be a nonnegative integer. To prove the main result in this section, we introduce the class $\mathcal{G}^s(C)$, consisting of all $f\in \mathcal{C}^{s}(\R^d)$ such that 
\begin{align}\nonumber
 &|f(x)|\leq C(1+\| x\|^{3}),\\
 \label{DefinitionGs}
    &|D^{\alpha}f(x)|\leq C(1+\| x\|^{s+1}), \ \ |\alpha|\leq s.
\end{align}}
Our next results gives an estimate of the complexity of this class, in terms of covering numbers with respect to the random metric $L_2(P_n)$. The proof can be easily adapted from the proof of Proposition 3 in \cite{Weed19}. We omit further details. 
\begin{Lemma} Assume $\mathcal{G}^s(C)$ is as above. If $X_1, \dots, X_n$ are i.i.d random variables with $\sigma^2$-subgaussian law $P\in \mathcal{P}(\R^d)$, $P_n$ is the associated empirical measure and $L=\frac 1 n\sum_{i=1}^n e^{-\|X_i\|^2/(4d\sigma^2)}$ then, for a constant $C_{s,d}$ depending only on $s$ and $d$,
\begin{equation}\label{DonskerG}
   \log \mathcal{N}(\epsilon,\mathcal{G}^s(C), L_2(P_n))\leq C_{s,d} L^{\frac{d}{2s}}\epsilon^{\frac{-d}{s}}(1+\sigma^{d})(1+\sigma^{s})^{\frac{d}{s}}.
\end{equation}
\end{Lemma}
Finally, we introduce the space $\mathcal{G}^s=\bigcup_{C\geq 0}\mathcal{G}^s(C)$ 
endowed with the norm $$\|f \|_s=\Big\|\frac{f }{1+\|\cdot\|^3}\Big\|_{\infty}+ \sum_{i=1}^s\sum_{|k|=i}\Big\|\frac{D^k f }{1+\|\cdot\|^{s+1}}\Big\|_{\infty}$$
Let $\big({\mathcal{G}^{s}}\big)'$ denote the dual space of $\mathcal{G}^{s}$, endowed with the {dual} norm
$$\| G\|_{s}'=\sup_{f\in\mathcal{G}^s, \ \| f\|_{s}\leq 1}|G(f)|.$$ 
With these ingredients we are ready to prove the main technical result of this section, from which we obtain the CLT for the entropic transportation cost with natural centering constants (Theorem \ref{NaturalCLT} below).
\begin{Lemma}\label{Lemma:opequena}
If $P,Q\in \mathcal{P}(\R^d)$ be $\sigma^2$-subgaussian probabilities, then 
\begin{align}\label{one_sample}
    \sqrt{n}\left|\E S(P_n,Q)-S(P,Q)\right|\rightarrow 0.
\end{align}
Moreover, if $m=m(n)$ and $\lambda:=\lim_{n\rightarrow\infty}\frac{n}{n+m}\in (0,1)$, then 
\begin{align}
    \sqrt{\frac{nm}{n+m}}\left|\E S(P_n,Q_m)-S(P,Q)\right|\rightarrow 0.
\end{align}
\end{Lemma}
\begin{proof}
Let $(f_n,g_n)\in L_1(P_n)\times L_1(Q)$ be the unique pair of optimal potentials satisfying \eqref{optimalli} and \eqref{normalize} for $P_n,Q$. As noted above, by Proposition 6 in \cite{Weed19}, this pair satisfies \eqref{bound1} and  \eqref{bound2}. We observe that, by optimality of the potentials, 
\begin{equation*}
   S(P,Q)\geq  \int_{\R^d} f_n(x) d P(x)+\int_{\R^d} g_n(y) d Q(y)-\int_{\R^d\times \R^d}e^{{f_n(x)+g_n(y)- \frac{1}{2}\|x-y\|^2}} d P(x)dQ(y)+1,
\end{equation*}
which yields
\begin{align*}
  0\leq \sqrt{n}\left(\E S(P_n,Q)-S(P,Q)\right)\leq &    \E \int_{\R^d} f_n(x) \sqrt{n}(d P_n-dP)(x)\\ &-\E \int_{\R^d\times \R^d}e^{{f_n(x)+g_n(y)- \frac{1}{2}\|x-y\|^2}}  \sqrt{n}(d P_n-dP)(x)dQ(y).
\end{align*}
Now the optimality condition
\begin{equation*}
    \int_{\R^d\times \R^d}e^{{f_n(x)+g_n(y)- \frac{1}{2}\|x-y\|^2}}dQ(y) =1, \ \ \text{for all } x\in \R^d
\end{equation*}
 implies that 
\begin{equation*}
   0\leq \sqrt{n}\left(\E S_{1}(P_n,Q)-S_{1}(P,Q)\right)\leq     \E \int_{\R^d} f_n(x) \sqrt{n}(d P_n-dP)(x).
\end{equation*}
Set $s=[d/2]+1$ and let $(f^*,g^*)\in L_1(P)\times L_1(Q)$ be the unique pair of optimal potentials satisfying \eqref{optimalli} and \eqref{normalize} for $P,Q$.
Since $\E \int_{\R^d} f^*(x)\sqrt{n}(d P_n-dP)(x)=0$,
\begin{align*}
  0\leq \sqrt{n}\left(\E S_{1}(P_n,Q)-S_{1}(P,Q)\right)\leq &    \E \int_{\R^d} \{f_n(x)-f^*(x)\}\sqrt{n}(d P_n-dP)(x)
\end{align*}
We write now $\mathbb{G}_n$ for the empirical process indexed by $\mathcal{G}^s$, that is, $\mathbb{G}_n(f)=\sqrt{n}(P_n(f)-P(f))$, $f\in \mathcal{G}^s$, and note that
\begin{align*}
    |\mathbb{G}_n(f)|&\leq \sqrt{n}(P_n+P)(|f|)
    \leq \Big\|\frac{f }{1+\|\cdot\|^3}\Big\|_{\infty}\sqrt{n}(P_n+P)(1+\|\cdot\|^3)\\
    &\leq \sqrt{n}(P_n+P)(1+\|\cdot\|^3) \| f\|_s
    \leq \sqrt{n}(2+48(d\tilde{\sigma}^2)^3)\| f\|_s,
\end{align*}
where the last inequality comes from Lemma~1 in \cite{Weed19}. Consequently, we deduce that $\mathbb{G}_n$ belongs to the dual space $\big({\mathcal{G}^s}\big)'$, for all $n\in \N$, and we get the bound
 \begin{align*}
   \sqrt{n}\left(\E S_{1}(P_n,Q)-S_{1}(P,Q)\right)\leq &    \E \left\lbrace  \| \mathbb{G}_n\|'_{s}\|f^*-f_n\|_{s}\right\rbrace.
\end{align*}
Using Cauchy–Schwarz's inequality we see that
 \begin{align}\label{eq:terms}
   \sqrt{n}\left(\E S_{1}(P_n,Q)-S_{1}(P,Q)\right)\leq &   \sqrt{ \E {\| \mathbb{G}_n\|'_{s}}^2 \E \| f^*-f_n\|^2_{s}}.
\end{align}
Note that $ \| \mathbb{G}_n\|'_{s}$ is the sup taken on the unit ball of $\mathcal{G}^s$, which is {contained in}  $\mathcal{G}^s(1)$. We can conclude, by using \eqref{DonskerG} and  Theorem 3.5.1 and Exercise 2.3.1 in \cite{Gin2015MathematicalFO}, that there exists a constant $C_{s, d}>0$ such that
 \begin{align*}
     \E {\| \mathbb{G}_n\|'_{s}}^2&\leq C_{s,d}\E \left(\int_{0}^{\max_{\| f\|_{s}\leq 1}  \| f\|_{L_2(P_n)}} \sqrt{L^{\frac{d}{2s}}\epsilon^{\frac{-d}{s}}(1+\sigma^{d})(1+\sigma^{s})^{\frac{d}{s}}}d\epsilon\right)^2\\
     &\leq (1+\sigma^{d})(1+\sigma^{s})^{\frac{d}{s}} C_{s, d}\E \left(\int_{0}^{  1+ 4\sqrt{3}d^{3/2}\tilde{\sigma}^3}{L^{\frac{d}{4s}}\epsilon^{\frac{-d}{2s}}}d\epsilon\right)^2\\
     &\leq C'_{s, d}(1+\sigma^{2d}) \E  L^{\frac{d}{2s}}{( 1+ \tilde{\sigma}^3)}^{\frac{2s-d}{s}},
 \end{align*}
 where we have used first Lemma~1 in \cite{Weed19} to bound 
 $$ \max_{\| f\|_{s}\leq 1}  \| f\|_{L_2(P_n)}\leq 1+\left(\int \|x\|^6 dP_n(x)\right)^{1/2} \leq 1+ 4\sqrt{3}d^{3/2}\tilde{\sigma}^3 $$
 and then the fact that $s=[d/2]+1 $. Using the Cauchy-Schwarz inequality we see that
 \begin{align*}
     \E  L^{\frac{d}{2s}}{( 1+ \tilde{\sigma}^3)}^{\frac{2s-d}{s}}\leq \sqrt{\E  L^{\frac{d}{s}}\E {( 1+ \tilde{\sigma}^3)}^{\frac{2(2s-d)}{s}}}
 \end{align*}
 where we can use the fact that $\E L < C$ for a positive constant $C$ independent of $n$ and Lemma~\ref{Lemma_weed} to conclude that
 $\lim\sup \E {\| \mathbb{G}_n\|'_{s}}^2<\infty$. 
\\ \\
To deal with the second term in \eqref{eq:terms} we denote $\Delta_n=f^*-f_n$. We prove next that $\|\Delta_n\|_s\rightarrow 0$ almost surely, and then that it is dominated by a random variable with finite second moment. Together, these facts imply that  $\E \| f^*-f_n\|^2_{s}\rightarrow 0$  and conclude the proof. The first claim is given by the following result.
\begin{Lemma}\label{lemma:boundinGs}
Let $P,Q\in \mathcal{P}(\R^d)$ be $\sigma$-subgaussian probabilities, and $P_n, Q_n$ associated empirical measures. Then, the optimal transport potentials $(f_n,g_n)$ for $P_n,Q_n$ satisfy
 $\|f_n-f^*\|_s\rightarrow 0$ and $\|g_n-g^*\|_s\rightarrow 0$ almost surely.
\end{Lemma}

\begin{proof}
We prove the result for $f_n$, with the same conclusion following for $g_n$ by symmetry. First, we use induction to prove convergence of the derivatives up to order $s$. We follow classical arguments in real analysis, see \cite{Rudin1987RealAC}:\\ \\
 \textbf{(1)} \emph{ For $ J=0$}, Proposition 4 in \cite{Weed19} shows that, almost surely, $\Delta_n:=f^*-f_n\rightarrow 0$ uniformly in compact sets.\\ \\
 \textbf{(2)}\emph{ Assume that for every  $k$ with $ |k|\leq J-1$, we have $D^k\Delta_n\rightarrow 0$, uniformly in compact sets.} Let $k=(k_1,\dots, k_{d})$ be such that $|k|= J$ and let $B_R\subset \R^d$ be the ball of radius $R$ centered at $0$. Using the fact that all the derivatives of $ D^k\Delta_n $ are bounded and $ D^k\Delta_n $ is itself pointwise bounded, see Proposition 1 and Lemma 2 in \cite{Weed19}, we derive that the sequence $ D^k\Delta_n $ is equicontinuous and  bounded for all points. We can then apply  the Arzel\`a-Ascoli theorem on $B_R$ to deduce that, up to subsequences, $ D^k\Delta_n \rightarrow \Delta^k$ uniformly on $B_R$. Suppose, without losing generality, that $k_1\geq 1$, set $k'=(k_1-1,\dots, k_d)$ and note that
 $$ D^{k'}\Delta_n(x)=\int_{0}^{x_1} D^k\Delta_n(t, ,x_2,\dots,x_d)dt + D^{k'}\Delta_n(0,x_2,\dots,x_d), $$
which implies that 
\begin{align*}
   & |D^{k'}\Delta_n(x)-\int_{0}^{x_1} \Delta^k(t, ,x_2,\dots,x_d)dt |\\
    & \leq \int_{0}^{x_1} |D^k\Delta_n(t, ,x_2,\dots,x_d)-\Delta^k(t, ,x_2,\dots,x_d)|dt+ |D^{k'}\Delta_n(0,x_2,\dots,x_d)|. 
\end{align*}
As a consequence,
\begin{align*}
     &\sup_{x\in B_R}|D^{k'}\Delta_n(x)-\int_{0}^{x_1} \Delta^k(t, ,x_2,\dots,x_d)dt |\\
     &\leq R \sup_{x\in B_R}|D^k\Delta_n(x)- \Delta^k(x)|+ |D^{k'}\Delta_n(0,x_2,\dots,x_d)|\rightarrow 0,
\end{align*}
where the limit follows from the induction hypothesis (recall that $\sup_{x\in B_R}|D^{k'}\Delta_n(x)|$ $\rightarrow 0$). By uniqueness of the limit we conclude that $0=\int_{0}^{x} \Delta^kdx_1$, which implies that $ \Delta^k=0$. By taking $R\rightarrow\infty$ we conclude that $D^k\Delta_n\rightarrow 0$ uniformly on the compact sets of $\R^d$. 

To show convergence in the norm $\|\cdot\|_s$, it suffices to show that for any $\epsilon > 0$, there exists an $n_0$ such that $\|\Delta_n\|_s \leq \epsilon$ for all $n \geq n_0$.
Recall that by Lemma~\ref{Lemma_weed} (i) and Proposition~1 in \cite{Weed19}, there exists an almost surely finite random variable $\tilde \sigma$ and a constant $K_{s, d}$ such that for all $n \in \N$ and $x \in \R^d$,
\begin{equation}\label{norm_bound}
\begin{split}
	\frac{|\Delta_n(x)|}{1 + \|x\|^3} & \leq K_{s, d} \frac{1 + \tilde \sigma^4}{1 + \|x\|} \\
	\frac{|D^k \Delta_n(x)|}{1 + \|x\|^{s+1}} & \leq K_{s, d} \frac{1 + \tilde \sigma^{3s}}{1 + \|x\|} \quad \quad \forall |k| \leq s.
\end{split}
\end{equation}
We obtain that there exists a finite random variable $\tilde K$ such that
\begin{equation*}
	\frac{|\Delta_n(x)|}{1 + \|x\|^3} + \sum_{i=1}^s \sum_{|k|=i} \frac{|D^k \Delta_n(x)|}{1 + \|x\|^{s+1}} \leq \epsilon/2 \quad \quad \forall \|x\| > \tilde K \epsilon^{-1}, n \geq 0\,.
\end{equation*}
Since $\Delta_n$ and $D^k \Delta_n$ converge uniformly to zero on the compact set $\{x \in \R^d: \|x\| \leq c_{d, s} \tilde K \epsilon^{-1}\}$, there exists an $n_0$ for which
\begin{equation*}
	\frac{|\Delta_n(x)|}{1 + \|x\|^3} + \sum_{i=1}^s \sum_{|k|=i} \frac{|D^k \Delta_n(x)|}{1 + \|x\|^{s+1}} \leq \epsilon/2 \quad \quad \forall \|x\| \leq \tilde K \epsilon^{-1}, n \geq n_0\,.
\end{equation*}
Combining these claims, we obtain that $\|\Delta_n\|_s \leq \epsilon$ for all $n \geq n_0$, as desired.
\end{proof}
To complete the proof of Lemma \ref{Lemma:opequena} it only remains to prove that $\|f_n-f^*\|_s$ can be dominated by a random variable with finite second moment.
We have from \eqref{norm_bound} above that $\|\Delta_n\|_s^2 \leq K'_{s, d} (1+\tilde \sigma^{3s})^2$ for some constant $K_{s, d}'$.
It only remains to show that $\E \left( 1+\tilde{\sigma}^{3s}\right)^2<\infty$. But Lemma~\ref{Lemma_weed} implies that all moments of $\tilde \sigma$ are finite, which completes the proof.
%

To deal with the two-sample case, we split the difference as follows
\begin{align*}
    &\sqrt{\frac{nm}{n+m}}|\E S_{1}(P_n,Q_m)-S_{1}(P,Q)|\\
    &\leq \sqrt{\frac{nm}{n+m}}|\E S_{1}(P_n,Q_m)-S_{1}(P,Q_m)| +\sqrt{\frac{nm}{n+m}}|\E S_{1}(P,Q_m)-S_{1}(P,Q)|.
\end{align*}
The second term tends to $0$ by using \eqref{one_sample}. For the first one we denote $g_{n,m}$ a potential of $S_{1}(P_n,Q_m)$ and $g_{m}$ a potential of $S_{1}(P,Q_m)$. Applying \eqref{eq:terms} we derive
\begin{align*}
    \sqrt{m}|\left(\E S(P_n,Q_m)-S(P,Q_m)\right)|\leq &   \sqrt{ \E {\| F_m\|'_{s}}^2 \E \| g_{n,m}-g_{m}\|^2_{s}}\\
    &\leq   2\sqrt{ \E {\|F_m\|'_{s}}^2 \left(\E \| g_{n,m}-g^*\|^2_{s}+\E \|g_{m}-g^*\|^2\right)}.
\end{align*}
We conclude using Lemma~\ref{lemma:boundinGs} (which can be trivially adapted to this setup) and the subsequent argument. 
\end{proof}
As a consequence of Lemma \ref{Lemma:opequena} by simply considering the change of variables  $x\mapsto x \epsilon^{-\frac{1}{2}}$ (recall the comments at the beginning of this section) we obtain the generalization to any $\epsilon>0$.
\begin{Corollary}\label{Corollary2.5}
Let $P,Q\in \mathcal{P}(\R^d)$ be $\sigma$-subgaussian probabilities and $P_n$, $Q_m$ associated empirical measures. Then 
\begin{align*}
    \sqrt{n}\left|\E S_{\epsilon}(P_n,Q)-S_{\epsilon}(P,Q)\right|\rightarrow 0
\end{align*}
and 
\begin{align*}
    \sqrt{\frac{nm}{n+m}}\left|\E S_{\epsilon}(P_n,Q_m)-S_{\epsilon}(P,Q)\right|\rightarrow 0.
\end{align*}
\end{Corollary}

As announced, Corollary \ref{Corollary2.5} improves over Corollary 1 in \cite{Weed19}, which implied $|\E S_{\epsilon}(P_n,Q)-S_{\epsilon}(P,Q)|=O(n^{-1/2})$ rather than $\left|\E S_{\epsilon}(P_n,Q)-S_{\epsilon}(P,Q)\right|=o(n^{-1/2})$.
\begin{Remark}\label{fast_rate}
	In some cases we can go much further in this direction. In fact, if $P$ and $Q$ are compactly supported then (see Theorem \ref{Theorem_Potential}  below)
	$$\E \| f_n-f^*\|_{\mathcal{C}^s(\Omega)}^2\leq \frac{c_\Omega}{n}$$
	for some constant $c_\Omega$. Plugging this into \eqref{eq:terms} and using again the fact that $\limsup_n \E {\| \mathbb{G}_n\|'_{s}}^2<\infty$
	we conclude that 
	\begin{equation}\label{improvedRate}
		\big|\E  S_{\epsilon}(P_n,Q)-S_{\epsilon}(P,Q)\big|\leq {\textstyle \frac{C_\Omega} n}
	\end{equation}
	for some constant $C_\Omega>0$. A similar conclusion holds for the two-sample problem. Whether this improved rate remains valid for general subgaussian
	probabilities is an open question.
\end{Remark}

The following central limit becomes a direct consequence of Theorem~3 in \cite{Weed19}, which shows that the fluctuations around the mean are asymptotically Gaussian, i.e.
\begin{equation}\label{CLTold}
    \sqrt{n}(S_{\epsilon}(P_n,Q)-\E S_{\epsilon}(P_n,Q))\xrightarrow{w} N(0,\operatorname{Var}_P(f^*_\epsilon)).
\end{equation}
Here $(f^*_{\epsilon},g^*_{\epsilon})$ are optimal potentials for  $S_\epsilon(P,Q)$. We observe that, while the pair of optimal potentials is not uniquely defined, it follows from the uniqueness of the minimizer in \eqref{kanto_entrop} that if $(\tilde{f}_\epsilon,\tilde{g}_\epsilon)$ is another pair of optimal potentials then $\tilde{f}_\epsilon=f^*_{\epsilon}+K$ $P$-a.s. for some constant $K$. Hence, $\operatorname{Var}_P(f^*_\epsilon)$ is well defined in the sense that it does not depend on the choice of optimal potential.

\begin{Theorem}\label{NaturalCLT}
Let $P,Q\in \mathcal{P}(\R^d)$ be $\sigma$-subgaussian probabilities, then 
\begin{equation*}
    \sqrt{n}(S_{\epsilon}(P_n,Q)-S_{\epsilon}(P,Q))\xrightarrow{w} N(0,\operatorname{Var}_P(f^*_{\epsilon})),
\end{equation*}
where $(f^*_{\epsilon},g^*_{\epsilon})$ are optimal potentials for $S_\epsilon(P,Q)$. Moreover, if $\lambda:=\lim_{n,m\rightarrow\infty}\frac{n}{n+m}\in (0,1)$, 
\begin{equation*}
    \sqrt{\frac{nm}{n+m}}(S_{\epsilon}(P_n,Q_m)-S_{\epsilon}(P,Q))\xrightarrow{w} N(0,(1-\lambda)\operatorname{Var}_P(f^*_{\epsilon})+\lambda \operatorname{Var}_Q(g^*_{\epsilon})).
\end{equation*}
\end{Theorem}

One important advantage of Theorem \ref{NaturalCLT} over \eqref{CLTold} is that it can be exploited for inferential purposes. For instance, it enables to build confidence intervals for $S_\epsilon(P,Q)$ as follows.\\ Note that we can estimate the asymptotic variance in the one-sample CLT by
\begin{equation}\label{estimated_variance}
\hat{\sigma}_n^2:=\operatorname{Var}_{P_n}(f_{n,\epsilon})={ \frac 1 n}\sum_{i=1}^n f_{n,\epsilon}^2(X_i)-\Big({ \frac 1 n}\sum_{i=1}^n f_{n,\epsilon}(X_i)\Big)^2,
\end{equation}
where $(f_{n,\epsilon},g_{n,\epsilon})$ is a pair of optimal potentials for $S_\epsilon(P_n,Q)$. It follows from the proof of Lemma \ref{Lemma:opequena} that $\E \|f_{n,\epsilon}-f^*_{\epsilon}\|_s^2\to 0$. Hence, $\|f_{n,\epsilon}-f^*_{\epsilon}\|_s\to 0$ in probability. Using the elementary bound $|a^2-b^2\leq |a-b|^2+2|b||a-b|$ we see that that
$$\Big|\frac 1 n \sum_{i=1}^n f_{n,\epsilon}^2(X_i)-\frac 1 n \sum_{i=1}^n {f^*_{\epsilon}}^2(X_i)\Big|\leq (\|f_{n,\epsilon}-f^*_{\epsilon}\|_s^2+2 
\|f^*_{\epsilon}\|_s\|f_{n,\epsilon}-f^*_{\epsilon}\|_s){\textstyle \frac 1 n} \sum_{i=1}^n (1+\|X_i\|^3).$$
Since $\frac 1 n \sum_{i=1}^n {f^*_{\epsilon}}^2(X_i)\to \E _P({f^*_{\epsilon}}^2)$ a.s. we conclude that $\frac 1 n \sum_{i=1}^n f_{n,\epsilon}^2(X_i)
\to \E _P({f^*_{\epsilon}}^2)$ in probability and, arguing similarly for $\frac 1 n \sum_{i=1}^n f_{n,\epsilon}(X_i)$, that 
$$\hat{\sigma}_n^2\to \operatorname{Var}_P(f^*_{\epsilon})\quad \text{in probability}.$$
We conclude that 
$$   \frac{\sqrt{n}}{\hat{\sigma}_n}(S_{\epsilon}(P_n,Q)-S_{\epsilon}(P,Q))\xrightarrow{w} N(0,1)$$
and, as a consequence, that, writing $z_\beta$ for the $\beta$ quantile of the standard normal distribution,
\begin{equation}\label{ref3.10}
\left[S_{\epsilon}(P_n,Q)\pm \frac{\hat{\sigma}_n}{\sqrt{n}} z_{1-\alpha/2} \right] 
\end{equation}
is a confidence interval for $S_\epsilon(P,Q)$ of asymptotic level $1-\alpha$. A similar confidence interval can be constructed from the two-sample statistic. Such results will be illustrated in the Simulations Section.

\section{Convergence rates for optimal potentials}\label{section3}
The goal of this section is to prove a bound on the difference between empirical potentials and their population counterparts. In this section we assume that both measures, $P,Q,$ are supported in a compact set $\Omega\subset\R^d$. By translation invariance of the optimal transport problem, we may assume without loss of generality that $0 \in \Omega$. We write $D_\Omega$ for the diameter of $\Omega$ and let $(f^*,g^*)$ be a pair of optimal potentials (maximizers of \eqref{dual_entrop} for $P$ and $Q$) and $(f_n,g_n)$ their empirical counterpart (maximizers of \eqref{dual_entrop} for $P_n$ and $Q$). As noted above, these optimal potentials are unique up to an additive constant. 
In this section, we adopt the following normalization convention:
\begin{equation}
    \label{eq:expected0}
    \int g^*(y)dQ(y)=  \int g_n(y)dQ(y)=0,
\end{equation}
We show below that derivatives af the optimal potentials are uniformly bounded (see Lemma \ref{Lemma_geneway}). Additionally, the choice of optimal potentials in \eqref{eq:expected0} allows to control uniformly the optimal potentials, as we show in Lemma \ref{Lemma:control_unifrom}.
These are key ingredients for the aim of the section, namely, showing  
that the convergence rate of  $f_n$ (resp. $g_n$) towards $f^*$ (resp. $g^*$), is $O_p\left(\frac{1}{\sqrt{n}}\right)$.

The optimal potentials belong to the space  $\mathcal{C}^s(\Omega)$, for $s=\left\lceil \frac{d}{2}\right\rceil+1$, in which we consider the  norm
$$\|f \|_{\mathcal{C}^s(\Omega)}= \sum_{i=0}^s\sum_{|\alpha|= i}\|D^{\alpha} f \|_{\infty}.$$
In this section, we use the notation $c_{d, s}, c'_{d, s}, \dots$ to indicate unspecified positive constants depending on $d$ and $s$ whose value may change from line to line.
The optimality conditions~\eqref{optimalli2} imply the following bounds (see Proposition 1 in \cite{genevay2019}).
 \begin{Lemma}\label{Lemma_geneway}
 Let $\Omega \subset\R^d$ be a compact set and $P,Q\in \mathcal{P}(\Omega)$. Then the optimal potentials $(f^*, g^*)$ satisfy: 
 \begin{enumerate}
    \item[(i)]{$\min_{y\in \Omega}\{\frac{1}{2}\|x-y \|^2-g^*(y)\}\leq f^*(x)\leq \max_{y\in \Omega}\{\frac{1}{2}\|x-y \|^2-g^*(y)\}$},
    \item[(ii)] $f^*(x)$ is $D_\Omega$-Lipschitz,
    \item[(iii)] $f^*\in \mathcal{C}^{\infty}(\Omega)$ and $\|D^{\alpha}f^*\|_{\infty}\leq C_{d, \alpha} (1 + D_\Omega^{|\alpha|})$ for all multi-indices $\alpha$ with $|\alpha| \geq 1$, for some constant $C_{d, \alpha}$ depending only on $d$ and $\alpha$.
    \end{enumerate}
\end{Lemma}
\begin{proof}
	The first two claims are proven in Proposition 1 of \cite{genevay2019}, so it suffices to consider the last claim. We prove it for $f^*$, the case of $g^*$ being similar. Define $\Bar{f}^*(x)=f^*(x)-\frac{1}{2}\|x\|^2$.
	As in the proof of Proposition 1 in \cite{Weed19}, the Faà di Bruno formula yields
\begin{align*}
   -D^\alpha \overline{f}^*(x)= \sum_{\beta_1 + \dots + \beta_k = \alpha}\lambda_{\alpha,\beta_1, \dots, \beta_k}\prod_{j=1}^k \mu_{\beta_j, g}, \ \ \text{for all $x\in \Omega$,}
\end{align*}
where $\lambda_{\alpha, \beta_1, \dots, \beta_k}$ are combinatorial quantities and for a multi-index $\beta$ we define
 \begin{align*}
 	 \mu_{\beta, g}& =\frac{\int y^{\beta}e^{g^*(y)-\frac{1}{2}\|y\|^2+\langle x,y \rangle}dQ(y)}{\int e^{g^*(y)-\frac{1}{2}\|y\|^2+\langle x,y \rangle}dQ(y)}\\
 	 & =\int y^{\beta}e^{h^*(x,y)-\frac{1}{2}\|x-y\|^2}dQ(y) \\
 	 & =\int \prod_{i=1}^d y_i^{\beta_i}e^{h^*(x,y)-\frac{1}{2}\|x-y\|^2}dQ(y).
 \end{align*}
By the optimality condition~\eqref{optimalli2}, $\int e^{h^*(x,y)-\frac{1}{2}\|x-y\|^2}dQ(y) = 1$.
As a consequence, there exists $C'_{d, \alpha}$ such that
$   \|D^\alpha \overline{f}^*\|_\infty \leq C'_{d,\alpha} D_{\Omega}^{|\alpha|}.
$
Since $\|D^\alpha \frac 12 \|x\|^2\|_\infty \leq 1 + D_\Omega$ for $|\alpha| \geq 1$, we obtain
$
   \|D^a {f}^*\|\leq C'_{d,\alpha} D_{\Omega}^{\alpha}+1 + D_\Omega \leq C_{d, \alpha} (1+D_\Omega^{|\alpha|})
$.
\end{proof}

\begin{Remark}
Since the probabilities $P_n$ and $Q_n$ are also supported on the same compact set $\Omega$,  Lemma~\ref{Lemma_geneway} holds also for $f_n$ and $g_n$.
\end{Remark}

We also obtain bounds on the derivatives of $e^{h^*(x, y) - \frac 12 \|x - y\|^2}$.
\begin{Lemma}\label{Lemma:appendix2}
	For any multi-index $\beta$, the function $x \mapsto x^\beta e^{h^*(x, y) - \frac 12 \|x - y\|^2}$ has $\mathcal{C}^s(\Omega)$ norm at most $c_{d, s} e^{D_\Omega^2} (1 + D_\Omega^{s + |\beta|})$.
\end{Lemma}
\begin{proof}
	By Lemma~\ref{Lemma_geneway}, $\|x^\beta e^{h^*(x, y) - \frac 12 \|x - y\|^2}\|_\infty \leq D_\Omega^{|\beta|} e^{D_\Omega^2}$.
	For any $1 \leq |\alpha| \leq s$, the Faà di Bruno formula implies
	\begin{equation*}
		D^\alpha e^{h^*(x, y) - \frac 12 \|x - y\|^2} = e^{h^*(x, y) - \frac 12 \|x - y\|^2} \sum_{\lambda_1 + \dots + \lambda_s = \alpha} \gamma_{\alpha, \lambda_1,\dots, \lambda_s}\prod_{j=1}^s D^{\lambda_j} (h^*(x, y) - \frac 12 \|x - y\|^2)
	\end{equation*}
	for some combinatorial coefficients $\gamma_{\alpha, \lambda_1,\dots, \lambda_s}$, where the derivative operators are taken with respect to the $x$ variable.
	By Lemma~\ref{Lemma_geneway}, this quantity is bounded in magnitude by $c_{d, s} e^{D_\Omega^2} (1 + D_\Omega^{|\alpha|})$ for some constant $c'_{d, s}$.
	This implies
	\begin{equation*}
		|D^\alpha x^\beta  e^{h^*(x, y) - \frac 12 \|x - y\|^2}| \leq c_{d, s}'' e^{D_\Omega^2} (1 + D_\Omega^{s + |\beta|}) \quad \quad \text{for all $1 \leq |\alpha| \leq s$.}
	\end{equation*}
	Therefore, choosing $c_{d, s}$ to be a sufficiently large constant depending on $d$ and $s$ yields the claim.
\end{proof}

For our particular choice of optimal potentials we can also control the uniform norm, as follows.
\begin{Lemma}\label{Lemma:control_unifrom}
Under \eqref{eq:expected0}, we have
$$\|f^*\|_{\infty},\|f_n\|_{\infty},\|g^*\|_{\infty},\|g_n\|_{\infty} \leq {\textstyle\frac{1}{2}}D_{\Omega}^2.$$
\end{Lemma}

\begin{proof} Since $S_{\epsilon}(P,Q)=  \int_{\R^d} f^*(x) d P(x)+\int_{\R^d} g^*(y) d Q(y)\geq 0$, \eqref{eq:expected0} implies $\int_{\R^d} f^*(x) d P(x)\geq 0$. Therefore, using  first the optimality conditions, then Jensen's inequality and finally \eqref{eq:expected0}, we obtain 
\begin{align*}
    g^*(y)&=-\log\left(\int e^{{f^*(y)- {\textstyle \frac{1}{2}}\|x-y\|^2}}dP(y) \right)
    \leq \int \{ {\textstyle \frac{1}{2}}\|x-y\|^2-f^*(y)\}dP(y)\leq {\textstyle \frac{1}{2}}D_{\Omega}^2,
\end{align*}
for all $y\in\Omega$. By the same argument $f^*(x)\leq \frac{1}{2}D_{\Omega}^2,$ for all $x\in\Omega$. Set $x\in \Omega$, by Lemma~\ref{Lemma_geneway}, 
\begin{align*}
    f^*(x)\geq \min_{y\in \Omega}\{\textstyle\frac{1}{2}\|x-y \|^2-g^*(y)\}\geq -\max_{y\in \Omega}g^*(y)\geq -{\textstyle\frac{1}{2}}D_{\Omega}^2,
\end{align*}
and the same for $g^*$.
\end{proof}

{For any $a,b\in \mathcal{C}^s(\Omega)$ denote 
\begin{align*}
    L(a,b)=\int_{\R^d} a(x) d P(x)+\int_{\R^d} b(y) d Q(y)-\int_{\R^d\times \R^d}e^{{a(x)+b(y)- \frac{1}{2}\|x-y\|^2}} d P(x)dQ(y)+1
\end{align*}
and its semi-empirical counterpart
\begin{align*}
    L_n(a,b)=\int_{\R^d} a(x) d P_n(x)+\int_{\R^d} b(y) d Q(y)-\int_{\R^d\times \R^d}e^{{a(x)+b(y)- \frac{1}{2}\|x-y\|^2}} d P_n(x)dQ(y)+1.
\end{align*}
Let us denote by $h^*$ and $h_n$ the functions, belonging to $\mathcal{C}^s(\Omega\times\Omega)$, defined by
\begin{equation}\label{defin_h}
  \text{$h^*(x,y)=f^*(x)+g^*(y)$, $h_n(x,y)=f_n(x)+g_n(y)$,}  
\end{equation}
and $\pi^*\in \mathcal{P}(\Omega\times \Omega)$ the optimal coupling defined by  $d\pi^*=e^{h^*(x,y)-||x-y||^2} d P(x)dQ(y)$.
Abusing notation, we  write  $L(f_n,g_n)=L(h_n)$ and $L(f^*,g^*)=L(h^*)$.} 
As a consequence of Lemma~\ref{Lemma:control_unifrom} we obtain the following useful bound,
\begin{equation}
    \label{boundH}
    \text{ $\|h^* \|_{\infty}, \ \|h_n \|_{\infty}\leq  D_{\Omega}^2$, for all $n\in \N$}.
\end{equation}
At this point, we can state the main theorem of this section.
\begin{Theorem}\label{Theorem_Potential}
Let $\Omega\subset\R^d$ be a compact set and $P,Q\in \mathcal{P}(\Omega)$. Assume $(f^*,g^*)$ are optimal potentials for $P,Q$ and $(f_n,g_n)$ for $P_n,Q$ satisfying \eqref{eq:expected0}. Then there exists a constant  $C_{d}$, depending only on $d$, such that
\begin{align*}
        \E \| g_n-g^*\|_{\mathcal{C}^s(\Omega)}^2,\ \E \| f_n-f^*\|_{\mathcal{C}^s(\Omega)}^2\leq \textstyle \frac{C_{d}}{n} D_\Omega^{5(d+1)}e^{15 D_\Omega^2}
\end{align*}
Moreover,  if $(f_{n,m},g_{n,m})$ are optimal potentials for $P_n,Q_m$ satisfying \eqref{eq:expected0}, then 
\begin{align*}
        \E \| g_{n,m}-g^*\|_{\mathcal{C}^s(\Omega)}^2,\ \E \| f_{n,m}-f^*\|_{\mathcal{C}^s(\Omega)}^2\leq \textstyle \frac{C_{d}}{\min\{n, m\}} D_\Omega^{5(d+1)}e^{15 D_\Omega^2}. 
\end{align*}
\end{Theorem}

The proof is divided in a sequence of technical lemmas, of some independent interest. We show first (Lemma \ref{Lemma:cuadratic}) that the functional $L$ is well-behaved in the sense of being strongly concave near its maximum. Then (in Lemma \ref{lemma:lips}) we show that the functional $L_n-L$ is Lipsichtz. Typically, these two results are enough to prove convergence at the fast $n^{-1}$ rate (see, e.g., Theorem 3.2.5 of \cite{Vart_Well}).
Unfortunately, in our case, the norms appearing in Lemmas~\ref{Lemma:cuadratic} and \ref{lemma:lips} are different. This technical issue can be handled thanks to Lemmas \ref{Lemma:technical1} and \ref{Lemma_uniform_poten}.
\begin{Lemma}\label{Lemma:cuadratic}
Let $\Omega\subset\R^d$ be a compact set and $P,Q\in \mathcal{P}(\Omega)$, then 
\begin{align}\label{condition cuadratic}
    L(h_n)- L(h^*)&\leq -  {\textstyle \frac{1}{2}}\|h_n-h^* \|_{L^2(d\pi^*)}^2e^{-\|h_n-h^*\|_{\infty}},
\end{align}
where $h^*$, $h_n$ and $\pi^*$ are defined in \eqref{defin_h}.
\end{Lemma}

\begin{proof} The inequality $e^x\geq 1+x+ \frac{e^{-|x|}}{2}x^2$, which can be checked by elementary means, implies that 
\begin{eqnarray*}
\lefteqn{\int e^{{h_n(x,y)- \frac{1}{2}||x-y||^2}} d P(x)dQ(y)}\hspace*{1cm}\\
    &=&\int e^{h_n(x,y)-h^*(x,y)}e^{h^*(x,y)- \frac{1}{2}||x-y||^2} d P(x)dQ(y)\\
    &\geq & \int \left\lbrace 1+ (h_n(x,y)-h^*(x,y)) +{\textstyle\frac{1}{2}}(h_n(x,y)-h^*(x,y))^2e^{-|h_n(x,y)-h^*(x,y)|} \right\rbrace d\pi^*(x,y)\\
    &\geq & \int \left\lbrace 1+ (h_n(x,y)-h^*(x,y))+{\textstyle\frac{1}{2}}(h_n(x,y)-h^*(x,y))^2e^{-||h_n-h^*||_{\infty}}\right\rbrace d\pi^*(x,y).
\end{eqnarray*}
The optimality conditions yield $L(h^*)=\int h^*(x,y)d P(x)dQ(y)$. Hence,
\begin{align*}
            \int & e^{{h_n(x,y)- \frac{1}{2}||x-y||^2}} d P(x)dQ(y)
    \\
    &\geq -L(h^*)+ \int \left \lbrace 1+ h_n(x,y)+{\textstyle\frac{1}{2}}(h_n(x,y)-h^*(x,y))^2e^{-||h_n-h^*||_{\infty}}\right\rbrace d\pi^*(x,y).
\end{align*}
We conclude by  using the relation  $ \int h_n(x,y) d\pi^*= \int h_n(x,y)d P(x)dQ(y)$, which follows from the optimality conditions.
\end{proof}

\begin{Lemma}\label{lemma:lips}
Under the assumptions of Lemma~\ref{Lemma:cuadratic}, we have 
\begin{equation}
   L_n(h_n)-L(h_n)-L_n(h^*)+L(h^*)\leq \|P-P_n\|_{\mathcal{C}^s_1(\Omega)} \|h_n-h^*\|_{\mathcal{C}^s(\Omega^2)}, \ \ a.s.
\end{equation}
where 
\begin{equation}\label{def:dualnorm}
   \|P-P_n\|_{\mathcal{C}^s_1(\Omega)}:=\sup_{\|f\|_{\mathcal{C}^s(\Omega)}\leq 1 }\int f(x)(dP_n(x)-dP(x)).
\end{equation}
\end{Lemma}

\begin{proof} As noted above, the optimality conditions imply that
\begin{align*}
    L_n(h_n)&=\int h_n(x,y)dP_n(x)dQ(y),\quad L(h_n)=\int h_n(x,y)dP(x)dQ(y),\\
    L_n(h^*)&=\int h^*(x,y)dP_n(x)dQ(y),\quad L(h^*)=\int h^*(x,y)dP(x)dQ(y).
        \end{align*}
Therefore we have
\begin{align*}
     & L_n(h_n)-L(h_n)-L_n(h^*)+L(h^*)\\
      &=\int h_n(x,y)dQ(y)(dP_n(x)-dP(x))-\int h^*(x,y)dQ(y)(dP_n(x)-dP(x))\\
       &=\int ( h_n(x,y)-h^*(x,y))dQ(y)(dP_n(x)-dP(x))\\   
       &\leq\|h_n-h^*\|_{\mathcal{C}^s(\Omega)}\sup_{\substack{\|h\|_{\mathcal{C}^s(\Omega^2)}\leq 1 \\ h(x,y)=f(x)+g(y)}}\int h(x,y)dQ(y)(dP_n(x)-dP(x)).
\end{align*}
Note that
\begin{eqnarray*}
\lefteqn{\sup_{\substack{\|h\|_{\mathcal{C}^s(\Omega^2)}\leq 1 \\ h(x,y)=f(x)+g(y)}}\int h(x,y)dQ(y)(dP_n(x)-dP(x))}\hspace*{1cm}\\
&=& \sup_{\substack{\|h\|_{\mathcal{C}^s(\Omega^2)}\leq 1 \\ h(x,y)=f(x)+g(y)}}\int g(y)dQ(y)(dP_n(x)-dP(x))+\int f(x)dQ(y)(dP_n(x)-dP(x)).
\end{eqnarray*}
Since the first term is $0$ and the second is not affected by adding constant to $f$, we see that it equals
$$  \sup_{\|f \|_{\mathcal{C}^s(\Omega)}\leq 1 }\int f(x)(dP_n(x)-dP(x)). $$
\end{proof}

As anticipated, Lemma~\ref{lemma:lips} works with the norm $\|\cdot\|_{\mathcal{C}^s(\Omega^2)}^2$ and Lemma~\ref{Lemma:cuadratic} with $\|\cdot\|_{L^2(d\pi^*)}$. Both norms are different, but the next technical results show how these norms are related in the present setup.
\begin{Lemma}\label{Lemma:technical1}
Under the assumptions of Lemma~\ref{Lemma:cuadratic},
\begin{align*}
   & \| D^\alpha f^*- D^\alpha f_n\|_{\infty}^2 \leq c_{d,s} D_{\Omega}^{2 |\alpha|} \|h_n-h^*\|_{\infty}^2,\\
    &\| D^\alpha g^*- D^\alpha g_n\|_{\infty}^2 \leq c_{d,s}D_{\Omega}^{2|\alpha|} e^{ 2D_{\Omega}^2}\left(\|h_n-h^* \|_{\infty}^2+D_{\Omega}^{2s}\|P-P_n\|_{\mathcal{C}^s_1(\Omega)}^2\right),
\end{align*}
   for every multi-index $\alpha$, with $1\leq |\alpha|\leq s$.
\end{Lemma}

\begin{proof}
	We let $c_{d, s}$ denote a positive constant depending on $d$ and $s$ whose value may change from line to line.
	We note first that $f^*(x)-f_n(x)=\Bar f^*(x)-\Bar f_n(x)$, where
\begin{align}\label{def_f_bar}
    \begin{split}
        \Bar{f}^*(x)=f^*(x)-{\textstyle \frac{1}{2}}\|x\|^2 \ \ \text{and} \ \  \Bar{f}_n(x)=f_n(x)-{\textstyle \frac{1}{2}}\|x\|^2.
    \end{split}
\end{align}
As in the proof of Lemma~\ref{Lemma_geneway}, the Faá di Bruno formula implies
\begin{multline*}
	 D^\alpha f_n(x)-D^\alpha f^*(x)=  D^\alpha \overline{f}_n(x)-D^\alpha \overline{f}^*(x)\\
	=  \sum_{\beta_1 + \dots + \beta_s = \alpha} \lambda_{\alpha, \beta_1, \dots, \beta_s}\left(\prod_{j=1}^s \int y^{\beta_j}e^{h_n(x,y)-\frac{1}{2}\|x-y\|^2}dQ(y)-\prod_{j=1}^s \int y^{\beta_j}e^{h^*(x,y)-\frac{1}{2}
		\|x-y\||^2}dQ(y)\right).
\end{multline*}
Splitting the product, this last term equals
\begin{multline*}
	     \sum_{\beta_1 + \dots + \beta_s = \alpha} \lambda_{\alpha, \beta_1, \dots, \beta_s}\sum_{i=1}^s\prod_{j< i} \int y^{\beta_j}e^{h_n(x,y)-\frac{1}{2}\|x-y\|^2}dQ(y)\\ 
	     \prod_{j>i} \int y^{\beta_j}e^{h^*(x,y)-\frac{1}{2}\|x-y\|^2}dQ(y)  \int y^{\beta_i} e^{-\frac{1}{2}\|x-y\|^2}\{ e^{h_n(x,y)}- e^{h^*(x,y)}\}dQ(y).
\end{multline*}
Since $0 \in \Omega$, it follows that $|y^{\beta_j}| \leq D_{\Omega}^{|\beta_j|}$. Using that $ |e^x-e^y|\leq (e^y+e^{x})|x-y| $ we upper bound $|D^\alpha f_n(x)-D^\alpha f^*(x)|$ by
\begin{multline*}
	 D_{\Omega}^{|\alpha|} \sum_{\beta_1 + \dots + \beta_s = \alpha} |\lambda_{\alpha, \beta_1, \dots, \beta_s}|\sum_{i=1}^s \int (e^{{h^*(x,y)}-\frac{1}{2}\|x-y\|^2}+e^{{h_n(x,y)}-\frac{1}{2}\|x-y\|^2})| h_n(x,y)- h^*(x,y)|dQ(y) \\
	\leq 2s D_{\Omega}^{ |\alpha|}  \sum_{\beta_1 + \dots + \beta_s = \alpha} |\lambda_{\alpha, \beta_1, \dots, \beta_s}| \|h_n - h^*\|_\infty,
\end{multline*}
where we have used \eqref{optimalli2} to bound the integral.
We conclude that $ \|D^af_n(x)-D^a f^*(x)\|_{\infty}^2 \leq c_{d, s} D_\Omega^{2 |\alpha|}\|h_n - h^*\|_\infty^2$.

Turning to $g_n$ and $g_n^*$, we can argue similarly to obtain  
\begin{multline*}
	    |D^\alpha g_n(y)-D^\alpha g^*(y)|=  |D^\alpha \overline{g}_n(y)-D^\alpha \overline{g}^*(y)|\\
	\leq  \sum_{\beta_1 + \dots + \beta_s = \alpha} |\lambda_{\alpha, \beta_1, \dots, \beta_s}|\left(\prod_{j=1}^s \int x^{\beta_j}e^{h_n(x,y)-\frac{1}{2}\|x-y\|^2}dP_n(x)-\prod_{j=1}^s \int x^{\beta_j}e^{h^*(x,y)-\frac{1}{2}\|x-y\|^2}dP(x)\right)\\ 
	=\sum_{\beta_1 + \dots + \beta_s = \alpha} |\lambda_{\alpha, \beta_1, \dots, \beta_s}|\sum_{i=1}^s\prod_{j< i} \int x^{\beta_j}e^{h_n(x,y)-\frac{1}{2}\|x-y\|^2}dP_n(x)\prod_{j>i} \int x^{\beta_j}e^{h^*(x,y)-\frac{1}{2}\|x-y\|^2}dP(x) \cdot \\ \left( \int x^{\beta_j}e^{h_n(x,y)-\frac{1}{2}\|x-y\|^2} dP_n(x)-   \int x^{\beta_j}e^{h^*(x,y)-\frac{1}{2}\|x-y\|^2} dP(x)\right).
\end{multline*}

Note that
  \begin{multline*}
\Big| \int x^{\beta_j}e^{h_n(x,y)-\frac{1}{2}\|x-y\|^2} dP_n(x)-   \int x^{\beta_j}e^{h^*(x,y)-\frac{1}{2}\|x-y\|^2} dP(x)\Big|\\
\leq \Big| \int x^{\beta_j}(e^{h_n(x,y)-\frac{1}{2}\|x-y\|^2}-e^{h^*(x,y)-\frac{1}{2}\|x-y\|^2}) dP_n(x)\Big|+\Big|   \int x^{\beta_j}e^{h^*(x,y)-\frac{1}{2}\|x-y\|^2} (dP(x)-P_n(x))\Big|\\
 \end{multline*}
Since $\|h_n\|_\infty, \|h^*\|_\infty \leq D_\Omega^2$ by \eqref{boundH}, the first term can be bounded by $2D_\Omega^{|\beta_j|} e^{D_\Omega^2}\|h_n - h^*\|_\infty$.
For the other term observe that by Lemma~\ref{Lemma:appendix2}, the function
 $x\mapsto x^{\beta_j}e^{h^*(x,y)-\frac{1}{2}\|x-y\|^2} $
 belongs to $\mathcal{C}^s(\Omega)$, with norm at most $c_{d, s} e^{D_\Omega^2} (1 + D_\Omega^{s + |\beta_j|})$.
We conclude that there exists some constant $c_{d,s}$ such that
\begin{equation}
    \label{eq_bound_Process_emp}
    |\int x^{\beta_j}e^{h^*(x,y)-\frac{1}{2}\|x-y\|^2}(dP_n(x)-dP(x))|\leq c_{d,s} D_{\Omega}^{|\beta_j|+s} e^{ D_{\Omega}^2}\|P-P_n\|_{\mathcal{C}^s_1(\Omega)}.
\end{equation}
Combining the last two estimates we finally have
 $$
\|D^a g^*- D^a g_n \|_{\infty}^2
   \leq c_{d,s}D_{\Omega}^{2|\alpha|} e^{ 2D_{\Omega}^2}\left(\|h_n-h^* \|_{\infty}^2+D_{\Omega}^{2s}\|P-P_n\|_{\mathcal{C}^s_1(\Omega)}^2\right),$$
   which allows us to conclude.
 \end{proof}

Now we need to compare the norms $\| \cdot\|_{\infty}$ and $\| \cdot\|_{L^2(d\pi^*)}$.
We set $C=e^{-\frac{3}{2}D_{\Omega}^2}$, and note that $\eqref{eq:expected0}$ implies 
\begin{eqnarray*}
\lefteqn{\int (h_n(x,y)-h^*(x,y))^2  e^{h^*(x,y)- \frac{1}{2}\|x-y\|^2}d P(x)dQ(y)}\hspace*{2cm}\\
    &\geq & C\int (h_n(x,y)-h^*(x,y))^2 d P(x)dQ(y)\\
    &= & C \{\int (f_n(x)-f^*(x))^2 d P(x)+\int (g_n(y)-g^*(y))^2 d Q(y)\\ && + 2\int (f_n(x)-f^*(x)) (g_n(y)-g^*(y))d P(x)dQ(y)\}.
\end{eqnarray*}
Since the last term equals $0$, we obtain the bound
\begin{equation}
    \label{eq:boundfgh}
    \|h_n-h^*\|_{L^2(d\pi^*)}^2\geq e^{- \frac{3}{2}D_{\Omega}^2}\left(\|f_n-f^*\|_{L^2(dP)}^2+\|g_n-g^*\|_{L^2(dQ)}^2\right).
\end{equation}

Finally, we prove the last technical result, which relates the $L^2$ and $L^\infty$ norms for the  difference of the potentials.
\begin{Lemma}\label{Lemma_uniform_poten}
Under the assumptions of Lemma~\ref{Lemma:cuadratic}, we have
\begin{eqnarray*}
\lefteqn{\|f_n-f^*\|_{L^2(dP)}^2+ \|g_n-g^*\|_{L^2(dQ)}^2}\hspace*{2cm}\\
&\geq &{\textstyle \frac{1}{2}}e^{-2D_{\Omega}^2} \left(\|f_n-f^*\|_{\infty}^2+\|g_n-g^*\|_{\infty}^2\right)
   -c_{d, s}D_{\Omega}^{2s} \|P-P_n\|_{\mathcal{C}^s_1(\Omega)}^2.
\end{eqnarray*}
\end{Lemma}

\begin{proof} We will work separately with $f^*$ and $g^*$.
Fixing $x\in \Omega$,  
Jensen's inequality  yields
\begin{equation*}
    |e^{-f^*(x)}-e^{-f_n(x)}|^2\leq \int \left(|e^{g^*(y)-\frac{1}{2}\|x-y\|^2}-e^{g_n(y)-\frac{1}{2}\|x-y\|^2}|\right)^2dQ(y).
\end{equation*}
Now, the mean value theorem implies
\begin{equation*}
    |x-y |e^{\min\{x, y\}}\leq |e^{x}-e^{y}|\leq e^{\max\{x, y\}}|x-y|, \quad x,y\in \R,
\end{equation*}
yielding
\begin{align*}
    e^{-2\max\{\|f^*\|_{\infty}, \|f_n\|_\infty\}}|f_n(x)-f^*(x) |^2&\leq |e^{-f^*(x)}-e^{-f_n(x)}|^2\\
    &\leq  e^{2\max\{\|g^*\|_{\infty}, \|g_n\|_\infty\}}\|g_n-g^*\|_{L^2(dQ)}^2.
\end{align*}
Consequently, using Lemma~\ref{Lemma:control_unifrom}, we have proved that 
\begin{equation*}
    \|g_n-g^* \|_{L^2(dQ)}^2\geq e^{-2D_{\Omega}^2} \|f_n-f^*\|_{\infty}^2,
\end{equation*}
Now we deal with $\|g_n-g^*\|_{\infty}^2$. We fix $y\in \Omega$. By  the triangle inequality we have
\begin{align*}
    &|e^{-g^*(y)}-e^{-g_n(y)}|\\
    &\leq \int \left|e^{f^*(x)+\frac{1}{2}\|x-y\|^2}-e^{f_n(y)+\frac{1}{2}\|x-y\|^2}\right|dP(y)+\left|\int e^{{f}_n(x)+\frac{1}{2}\|x-y\|^2}(dP(x)-dP_n)\right|.
\end{align*}
Squaring both sides we see that
\begin{align*}
    &|e^{-g^*(y)}-e^{-g_n(y)}|^2\\
    &\leq 2 \int \left|e^{f^*(x)+\frac{1}{2}\|x-y\||^2}-e^{f_n(y)+\frac{1}{2}\|x-y\|^2}\right|^2dP(y)+2\left|\int e^{{f}_n(x)+\frac{1}{2}\|x-y\|^2}(dP(x)-dP_n)\right|^2.
\end{align*}
The first term is bounded by $2 e^{D_\Omega^2} \|f_n - f^*\|_{L^2(dP)}^2$ as in the previous case.
Repeating the arguments which led to the bound \eqref{eq_bound_Process_emp}, the second term is at most $c_{d, s} e^{D_\Omega^2} (1+D_\Omega^{2s})\|P-P_n\|_{\mathcal{C}^s_1(\Omega)}^2$.
Together, these estimates yield
\begin{align*}
    e^{-D_{\Omega}^2}\|g_n-g^*\|_{\infty}^2
    \leq 2 e^{D_{\Omega}^2}\|f_n-f^*\|_{L^2(dP)}^2+c_{d, s}D_{\Omega}^{2s} e^{D_{\Omega}^2}\|P-P_n\|_{\mathcal{C}^s_1(\Omega)}^2.
\end{align*}
We conclude by rearranging this inequality and combining it with the bound on $\|g_n -g^*\|_{L^2(dQ)}^2$ derived above.
\end{proof}

We are ready now for the proof of the main result in this section.

\begin{proof}[\textbf{Proof of Theorem~\ref{Theorem_Potential}}]
	Combining Lemma~\ref{Lemma:cuadratic},
\eqref{eq:boundfgh} and Lemma~\ref{Lemma_uniform_poten} we see that
\begin{align*}
        L(h^*) -L(h_n) &\geq {\textstyle \frac 1 2}e^{-\|h_n-h^*\|_{\infty}} {e^{- \frac{3}{2}D_{\Omega}^2}}\left(\|f_n-f^* \|_{L^2(dP)}^2+ \|g_n-g^* \|_{L^2(dQ)}^2\right)\\
        &\geq {\textstyle \frac 1 2} {e^{- \frac{7}{2}D_{\Omega}^2}}\left(\|f_n-f^* \|_{L^2(dP)}^2+ \|g_n-g^* \|_{L^2(dQ)}^2\right)\\
        & \geq {\textstyle \frac 1 2}{e^{- \frac{7}{2}D_{\Omega}^2}} \left(\frac{1}{2}e^{-2D_{\Omega}^2} \left(\|f_n-f^* \|_{\infty}^2+\|g_n-g^* \|_{\infty}^2\right)
   -c_{d, s}D_{\Omega}^{2s} \|P-P_n\|_{\mathcal{C}^s_1(\Omega)}^2\right).
\end{align*}
Moreover, since $\|f_n-f^* \|_{\infty}^2+\|g_n-g^* \|_{\infty}^2\geq \frac{1}{2}\|h_n-h^* \|_{\infty}^2,$
we obtain 
\begin{align*}
        L(h^*) -L(h_n) & \geq {\textstyle \frac 1 2}{e^{- \frac{7}{2}D_{\Omega}^2}} \left(\frac{1}{4}e^{-2D_{\Omega}^2} \|h_n-h^* \|_{\infty}^2
   -D_{\Omega}^{2s} \|P-P_n\|_{\mathcal{C}^s_1(\Omega)}^2\right).
\end{align*}
 Lemma~\ref{Lemma:technical1} implies the existence of some constant $c_{d,s}$ such that  
 $$\|h_n-h^* \|_{\infty}^2\geq \frac{1}{ c_{d,s}D_{\Omega}^{2s} e^{ 2D_{\Omega}^2}}\left(\|f_n-f^* \|_{\mathcal{C}^s(\Omega)}^2+\|g_n-g^* \|_{\mathcal{C}^s(\Omega)}^2\right)-D_{\Omega}^{2s}\|P-P_n\|_{\mathcal{C}^s_1(\Omega)}^2, $$
 which yields
 \begin{multline}\label{oneside}
         L(h^*) -L(h_n)  \geq   c_{d, s}e^{- \frac{15}{2}D_{\Omega}^2} D_\Omega^{-2s}\left(\|f_n-f^*\|_{\mathcal{C}^s(\Omega)}^2+\|g_n-g^* \|_{\mathcal{C}^s(\Omega)}^2\right)
   \\ -c'_{d, s}e^{- \frac{7}{2} D_\Omega^2} D_{\Omega}^{2s} \|P-P_n\|_{\mathcal{C}^s_1(\Omega)}^2.
\end{multline}
On the other hand,  Lemma~\ref{lemma:lips} yields 
\begin{align*}
    \|P-P_n\|_{\mathcal{C}_1^s(\Omega)} \left(\| f_n-f^*\|_{\mathcal{C}^s(\Omega)}+ \| g_n-g^*\|_{\mathcal{C}^s(\Omega)} \right)     & \geq  L_n(h_n)-L(h_n)-L_n(h^*)+L(h^*)
       \\
      & \geq L_n(h^*)-L(h_n)-L_n(h^*)+L(h^*)\\&=  L(h^*)-L(h_n)
\end{align*}
The previous bound and \eqref{oneside} yield
\begin{multline*}
	\|P-P_n\|_{\mathcal{C}_1^s(\Omega)} \left(\| f_n-f^*\|_{\mathcal{C}^s(\Omega)}+ \| g_n-g^*\|_{\mathcal{C}^s(\Omega)} \right) \geq \\ c_{d, s}e^{- \frac{15}{2}D_{\Omega}^2} D_\Omega^{-2s}\left(\|f_n-f^*\|_{\mathcal{C}^s(\Omega)}^2+\|g_n-g^* \|_{\mathcal{C}^s(\Omega)}^2\right)
	 -c'_{d, s}e^{- \frac{7}{2} D_\Omega^2} D_{\Omega}^{2s} \|P-P_n\|_{\mathcal{C}^s_1(\Omega)}^2.
\end{multline*}
 which, by using the inequality $(a+b)^2\leq 2 (a^2+b^2)$, implies
 \begin{multline*}
 	\sqrt{2}\|P-P_n\|_{\mathcal{C}_1^s(\Omega)} \left(\| f_n-f^*\|_{\mathcal{C}^s(\Omega)}^2+ \| g_n-g^*\|_{\mathcal{C}^s(\Omega)}^2 \right)^{1/2} \geq \\ c_{d, s}e^{- \frac{15}{2}D_{\Omega}^2} D_\Omega^{-2s}\left(\|f_n-f^*\|_{\mathcal{C}^s(\Omega)}^2+\|g_n-g^* \|_{\mathcal{C}^s(\Omega)}^2\right)
 	-c'_{d, s}e^{- \frac{7}{2} D_\Omega^2} D_{\Omega}^{2s} \|P-P_n\|_{\mathcal{C}^s_1(\Omega)}^2.
 \end{multline*}
Denoting 
$\Delta_n =\left(\| f_n-f^*\|_{\mathcal{C}^s(\Omega)}^2+ \| g_n-g^*\|_{\mathcal{C}^s(\Omega)}^2 \right)^{\frac{1}{2}},$
 we get 
\begin{align}
    \label{inequality_bruta}
    \|P-P_n\|_{\mathcal{C}_1^s(\Omega)} \Delta_n
       \geq  c_{d, s}e^{- \frac{15}{2}D_{\Omega}^2} D_\Omega^{-2s}\Delta_n^2
   -c'_{d, s}e^{- \frac{7}{2} D_\Omega^2} D_{\Omega}^{2s} \|P-P_n\|_{\mathcal{C}^s_1(\Omega)}^2.
\end{align}
From this we obtain
\begin{align}
    \begin{split}
        \label{solve_inequation}
    \Delta_n&\leq c_{d,s}D_{\Omega}^{2s} e^{ \frac{15}{2}D_{\Omega}^2} \left(\|P-P_n\|_{\mathcal{C}_1^s(\Omega)}+\sqrt{\|P-P_n\|_{\mathcal{C}_1^s(\Omega)}^2+e^{ -11 D_{\Omega}^2}\|P-P_n\|_{\mathcal{C}^s_1(\Omega)}^2}\right)\\
    &\leq c_{d,s}D_{\Omega}^{2s} e^{ \frac{15}{2}D_{\Omega}^2}\|P-P_n\|_{\mathcal{C}_1^s(\Omega)}.
    \end{split}
\end{align}

Next, we analyze $\|P-P_n\|_{\mathcal{C}_1^s(\Omega)}$. Theorem 3.5.1 and Exercise~2.3.1 in \cite{Gin2015MathematicalFO} imply that there exists a numerical constant $C$ such that
$$n\E \|P-P_n\|_{\mathcal{C}_1^s(\Omega)}^2 \leq C \E \left(\int_{0}^{\max_{\| f\|_{\mathcal C^s(\Omega)}\leq 1}  \| f\|_{L_2(P_n)}} \sqrt{\log\left(2{N}(\epsilon,\mathcal{C}_1^s(\Omega),\|\cdot \|_{\infty})\right)}d\epsilon\right)^2.$$
By Proposition 1.1.  in \cite{VANDERVAART199493},
 $$\log\left(2{N}(\epsilon,\mathcal{C}_1^s(\Omega), \|\cdot\|_{\infty})\right)\leq c_{s,d}D_{\Omega}^d \left(\frac{1}{\epsilon}\right)^{\frac{d}{s}}.$$
 Since $\|f\|_{L_2(P_n)} \leq \|f\|_{\mathcal C^s(\Omega)}$ almost surely, the choice $s=[d/2]+1 $ yields the bound
\begin{align}\label{boundempirical}
  n \E \|P-P_n\|_{\mathcal{C}_1^s(\Omega)}^2 \leq &  c_{s,d}D_{\Omega}^{d} \left(\int_{0}^{1}  \left(\frac{1}{\epsilon}\right)^{\frac{d}{d+1}}d\epsilon\right)^2= c_{s,d}D_{\Omega}^{d},
\end{align}
which completes the proof for the one-sample case. \\ 

The two-sample case can be handled with the same argument plus some minor modifications, as follows. Let  $f_{n,m}$ be the optimal potential for $P_n$ and $Q_m$. Then,
\begin{align*}
    \| f_{n,m}-f^*\|_{\mathcal{C}^s(\Omega)}^2&\leq 2\| f_{n}-f^*\|^2_{\mathcal{C}^s(\Omega)}+2\| f_{n,m}-f_n\|^2_{\mathcal{C}^s(\Omega)}.
\end{align*}
The first term can be controlled by~\eqref{solve_inequation}.
Moreover, observe that the derivation of~\eqref{solve_inequation} did not use any facts about the measure $Q$ apart from the fact that it is supporte don $\Omega$.
Since $P_n$ is also supporte din $\Omega$, this implies that $\| f_{n,m}-f_n\|_{\mathcal{C}^s(\Omega)}$ can also be controlled by~\eqref{solve_inequation}, so that the bound
\begin{align}\label{bound_for_fnn}
    \| f_{n,m}-f^*\|^2_{\mathcal{C}^s(\Omega)}
    &\leq 
    c_{d,s}D_{\Omega}^{2s} e^{ \frac{15}{2}D_{\Omega}^2}
    \left(\|P-P_n\|_{\mathcal{C}_1^s(\Omega)}^2+\|Q-Q_m\|_{\mathcal{C}_1^s(\Omega)}^2\right)
\end{align}
holds. This and \eqref{boundempirical} complete the proof.
\end{proof}

\section{Convergence rates for Sinkhorn divergences}\label{sec:divergences}
In this section, we develop faster convergence rates for the \emph{Sinkhorn divergence}.
The entropic transportation cost, $S_\epsilon(P,Q)$ is not symmetric in $P,Q$ and does not satisfy $S_\epsilon(P,P)=0$. These observations motivated the introduction of Sinkhorn divergences~\cite{GenPeyCut18}:
For probabilities $P,Q\in \mathcal{P}(\mathbb{R}^d)$ the quadratic Sinkhorn divergence is defined as 
$$D_{\epsilon}(P,Q)= S_{\epsilon}(P,Q)-{\textstyle \frac{1}{2}}\left(  S_{\epsilon}(P,P)+ S_{\epsilon}(Q,Q) \right).$$
Clearly, $D_{\epsilon}(P,Q)$ is symmetric in $P,Q$ and $D_\epsilon(P,P)=0$. In fact (see Theorem 1 in \cite{Feydy2019InterpolatingBO}),  
$D_{\epsilon}(P,Q)\geq 0$, with $D_{\epsilon}(P,Q)= 0$ if and only if $P=Q$, and for measures supported on a compact set, convergence in Sinkhorn distance is equivalent to weak convergence.
This makes the Sinkhorn divergence a suitable measure of dissimilarity in applications.

In this section we 
obtain rates of convergence for empirical Sinkhorn divergences. More precisely, we consider independent samples $X_1,\dots,X_n$, $Y_1,\ldots,Y_m$ of i.i.d. r.v.'s with law $P\in \mathcal{P}(\Omega)$ and associated empirical measures $P_n$ and $P_m'$, respectively. Since $P_n$ and $P_m'$ converge weakly to $P$, the Sinkhorn divergence satisifes $D_{\epsilon}(P_n,P_m')\to 0$  a.s. The main result of this section gives a rate for this convergence.
\begin{Theorem}\label{Theorem_divergences}
Assume $\Omega\subset\R^d$ is compact, $P\in \mathcal{P}(\Omega)$ and $P_n$ and $P'_m$ are empirical measures as above. Then there exist constants $c_{d}$ and $c_d'$, depending only on $d$, such that 
\begin{itemize}
\item[(i)] (one-sample case)
$$     \E  D_1(P_n,P) \leq \frac{c_{d} }{n} {D_{\Omega}^{\frac{3d}{2}+1} \frac{ 32}{(d+1)^2}e^{ \frac{19}{2}D_{\Omega}^2}}.$$
\item[(ii)] (two-sample case)
$$
      \E  D_1(P_n,P_m')
    \leq \frac{ c_{d}'}{\min\{n, m\}} {D_{\Omega}^{\frac{3d}{2}+1} \frac{ 32}{(d+1)^2} e^{ \frac{19}{2}D_{\Omega}^2}}.
$$
\end{itemize}
\end{Theorem}\label{divergence_thm}
\begin{proof} We deal first with the one-sample case. We denote by $(f_{n,n},g_{n,n})$ the optimal potentials for $S_{1}(P_n,P_n)$, set $h_{n,n}(x,y)=f_{n,n}(x)+g_{n,n}(y)$ and write $d\pi_{n,n}(x,y)=e^{h_{n,n}(x,y)-\frac{1}{2}\|x-y \|^2}$ for the optimal measure and, as in \eqref{defin_h}, we write $h^*$, $\pi^*$ for the corresponding objects in the case of $S_1(P,P)$ and $h_n$, $\pi_n$ in the case of $S_1(P_n,P)$.  Then we can write
\begin{align}\label{divemp}
    D_1(P_n,P)=\int h_{n}(x,y)d\pi_{n}(x,y)-{\textstyle \frac{1}{2}}\left( \int h_{n,n}(x,y)d\pi_{n,n}(x,y)+\int h^*(x,y)d\pi^*(x,y)   \right).
\end{align}
Moreover, using the optimality conditions, we have
\begin{align*}
    \int \left( h_{n}(x,y)-h^*(x,y)\right)d\pi^*(x,y)=L(h_{n})-L(h^*)\leq 0,
\end{align*}
and 
\begin{align*}
    \int \left( h_{n}(x,y)-h_{n,n}(x,y)\right)d\pi_{n,n}(x,y)=L_{n,n}(h_{n})-L_{n,n}(h_{n,n})\leq 0,
\end{align*}
where $L$ is defined as in the previous section and 
\begin{align*}
    L_{n,n}(h)=\int \left\lbrace h(x,y)-e^{{h(x,y)- \frac{1}{2}||x-y||^2}}+1\right\rbrace d P_n(x)dP_n(x).
\end{align*}
%
Therefore, from \eqref{divemp} we obtain 
\begin{align}\label{divempBound}
\begin{split}
     D_1(P_n,P)
    \leq \int h_{n}(x,y)d\pi_{n}(x,y)-{\textstyle\frac{1}{2}}\left( \int h_{n}(x,y)d\pi_{n,n}(x,y)+\int h_n(x,y)d\pi^*(x,y)   \right).
\end{split}
\end{align}
Note, moreover, that the upper bound in \eqref{divempBound} can be rewritten as
\begin{align}
    \nonumber
    \int f_{n}(x)dP_n(x)&+\int g_{n}(y)dP(y)\\
    \nonumber
   &- {\textstyle \frac{1}{2}}\left(\int f_{n}(x)dP_n(x)+\int g_{n}(y)dP_n(y)+\int f_{n}(x)dP(x)+\int g_{n}(y)dP(y)   \right)\\
    \nonumber
   &= {\textstyle \frac{1}{2}} \int \left(f_{n}(x)-g_{n}(x)\right)dP_n(x)+{\textstyle\frac{1}{2}} \int \left(g_{n}(x)-f_{n}(x)\right)dP(x)\\
    \nonumber
   &= {\textstyle \frac{1}{2}} \int \left(f_{n}(x)-g_{n}(x)\right)\left(dP_n(x)-dP(x)\right)\\
    \nonumber
   &= {\textstyle \frac{1}{2}} \int \left(f_{n}(x)-g^*(x)\right)\left(dP_n(x)-dP(x)\right)\\
   \label{rewritten}
   &+ {\textstyle \frac{1}{2}} \int \left(g^*(x)-g_{n}(x)\right)\left(dP_n(x)-dP(x)\right)\,,
\end{align}
where $(f_n, g_n)$ are optimal entropic potentials for $P_n, P$ and $(f^*, g^*)$ are optimal transport potentials for $(P, P)$, where adopt the normalization convention $\int g^*(y) dP(y) = \int g_n(y) dP(y) = 0$.
The symmetry of $S_1(P,P)$ and the uniqueness of the entropic potentials up to additive constants implies that $f^* = g^* + a$ for some constant $a \in \R$. Plugging this into \eqref{rewritten} we obtain from \eqref{divempBound} that
\begin{align}\label{divempBound2}
\begin{split}
	D_1(P_n,P)
    \leq  {\textstyle \frac{1}{2}} \|P-P_n\|_{\mathcal{C}_1^s(\Omega)}\left(\| f_{n}-f^*\|_{\mathcal{C}^s(\Omega)}+\| g_{n}-g^*\|_{\mathcal{C}^s(\Omega)} \right).
\end{split}
\end{align}
From \eqref{solve_inequation}, we obtain, for some constant $c_{d,s}$, the bound
\begin{align*}
   \left(\| f_{n}-f^*\|_{\mathcal{C}^s(\Omega)}+\| g_{n}-g^*\|_{\mathcal{C}^s(\Omega)} \right)&\leq 2 \left(\| f_n-f^*\|_{\mathcal{C}^s(\Omega)}^2+ \| g_n-g^*\|_{\mathcal{C}^s(\Omega)}^2 \right)^{\frac{1}{2}}\\
   &\leq c_{d,s}D_{\Omega}^{2s} e^{ \frac{15}{2}D_{\Omega}^2}\|P-P_n\|_{\mathcal{C}_1^s(\Omega)}.
\end{align*}
We conclude as in the proof of Theorem~\ref{Theorem_Potential}.

For the two sample case we can adapt the argument above without much effort. Indeed, observe that 
we can write
\begin{eqnarray}
\nonumber
\lefteqn{D_1(P_n,P'_m)=\int h_{n,m}(x,y)d\pi_{n,m}(x,y)}\hspace*{2cm}\\
\label{divemp2samp}
&&-{\textstyle \frac{1}{2}}\Big(  \int h_{n,n}(x,y)d\pi_{n,n}(x,y)+\int h_{m,m}(x,y)d\pi_{m,m}(x,y)\Big)
\end{eqnarray}
and argue as in \eqref{divempBound} to get
\begin{multline}\label{divempBoundTwosamp}
	   D_1(P_n,P'_m)\leq\int h_{n,m}(x,y)d\pi_{n,m}(x,y)\\-{\textstyle \frac{1}{2}}\left( \int h_{n,m}(x,y)d\pi_{n,n}(x,y)+\int h_{n,m}(x,y)d\pi_{m,m}(x,y)   \right).
\end{multline}

Now we can reuse the same arguments leading to \eqref{divempBound}---just replacing $P$ by $P'_m$---to upper bound $D_1(P_n,P'_m)$ by
\begin{align}\label{boundTwo}
   {\textstyle \frac{1}{2}}\int \left(f_{n,m}(x)-g^*(x)\right)\left(dP_n(x)-dP'_m(x)\right)+ {\textstyle \frac{1}{2}} \int \left(g^*(x)-g_{n,m}(x)\right)\left(dP_n(x)-dP'_m(x)\right).
\end{align}
Once again, since  $(f^*,g^*)$ agree up to an additive constant, \eqref{boundTwo} is equivalent to
 \begin{align*}
  {\textstyle \frac{1}{2}} \int \left(f_{n,m}(x)-f^*(x)\right)\left(dP_n(x)-dP'_m(x)\right)+ {\textstyle \frac{1}{2}} \int \left(g^*(x)-g_{n,m}(x)\right)\left(dP_n(x)-dP'_m(x)\right).
\end{align*}
Finally, the two sample case can be deduced directly from the following inequality and  
\begin{align}\label{divempBoundEMpTwosamples}
\begin{split}
      D_1(P_n,P)
    &\leq  \frac{1}{2} \|P_n-P'_m\|_{\mathcal{C}_1^s(\Omega)}\left(\| f_{n,m}-f^*\|_{\mathcal{C}^s(\Omega)}+\| g_{n,m}-g^*\|_{\mathcal{C}^s(\Omega)} \right)
    \\
   & \leq  \frac{1}{2}\left( \|P'_m-P\|_{\mathcal{C}_1^s(\Omega)}+\|P_n-P\|_{\mathcal{C}_1^s(\Omega)}\right)\left(\| f_{n,m}-f^*\|_{\mathcal{C}^s(\Omega)}+\| g_{n,m}-g^*\|_{\mathcal{C}^s(\Omega)} \right)
\end{split}
\end{align}
and \eqref{bound_for_fnn}, which yields 
\begin{align*}
    \|f_{n,m}-f^*\|_{\mathcal{C}^s(\Omega)}+\| g_{n,m}-g^*\|_{\mathcal{C}^s(\Omega)}
    &\leq c_{d,s}D_{\Omega}^{2s} e^{ \frac{15}{2}D_{\Omega}^2}\left(\|P-P_n\|_{\mathcal{C}_1^s(\Omega)}+\|P-P'_m\|_{\mathcal{C}_1^s(\Omega)}\right)
\end{align*}
for a constant $c_{d,s}$ depending on $d$ and $s$.
We conclude as above.
\end{proof}
\section{Implementation issues and empirical results}
In this section we provide details about the practical implementation and statistical performance of the two-sample analog of the confidence interval \eqref{ref3.10}. 

Recall from Theorem~\ref{NaturalCLT} that 
\begin{equation}
    \label{empirical_tcl}
    \sqrt{\frac{{nm}}{n+m}}\frac{1}{\sigma_{\epsilon, \lambda}(P,Q)}(S_{\epsilon}(P_n,Q_m)-S_{\epsilon}(P,Q))\xrightarrow{w} N(0,1),
\end{equation}
where  $\sigma_{\epsilon,\lambda}^2(P,Q) $ is the asymptotic variance of the two-sample case. This variance can be  consistently estimated by 
\begin{equation}\label{estimated_variance2}
\hat{\sigma}_{n,m}^2:=\text{\footnotesize $\frac{m}{n+m}\left({ \frac 1 n}\sum_{i=1}^n f_{n,\epsilon}^2(X_i)-\Big({ \frac 1 n}\sum_{i=1}^n f_{n,\epsilon}(X_i)\Big)^2\right)+\frac{n}{n+m}\left( {\frac 1 m}\sum_{i=1}^m g_{n,\epsilon}^2(Y_i)-\Big({ \frac 1 m}\sum_{i=1}^m g_{n,\epsilon}(Y_i)\Big)^2\right)$},
\end{equation}
where $(f_{n,\epsilon},g_{n,\epsilon})$ is a pair of empirical potentials.  Hence,  writing $z_\beta$ for the $\beta$ quantile of the standard normal distribution and arguing as in Section~\ref{sec:TCL}, we can conclude that the interval 
\begin{equation*}
CI^{n,m}_{\alpha}=\left[S_{\epsilon}(P_n,Q_m)\pm \hat{\sigma}_{n,m}\sqrt{\frac{n+m}{nm}} z_{1-\alpha/2} \right]
\end{equation*}
is an asymptotic confidence interval of level $1-\alpha$. 

We investigate here the finite sample performance of this confidence interval. We consider the scenario where  $P\sim N(0,I_d/2)$ and $Q\sim N(\left(         1 , \dots,
         1
   \right)^t,I_d/2)$.  The population entropy
regularized cost has a closed form for Gaussian measures (see \cite{Barrio2020TheSE}, \cite{janati2020entropic} or \cite{mallasto2021entropy}), which, for our case, is 
$$S_{\epsilon}(P,Q)= 2\,d-\frac{\epsilon}{2}\left(d\sqrt{1+\frac{4}{\epsilon^2} }-d\log\left(1+\sqrt{1+\frac{4}{\epsilon^2} }\right)+d\, \log(2)-d\right). $$
We focus on the case $n=m$, for several choices of $n=50,100,250,500,1000,5000$, and study the influence of the parameters $d$ and $\epsilon$   on the  rate  of convergence of the true confidence level to the nominal level $1-\alpha$, for $\alpha=0.05$. To approximate this true confidence level we use Monte Carlo simulation, with $1000$ replicates of the interval.  The results are reported on  Table~\ref{Table}. In particular, we  compute $CI^{n,n}_{0.05}$ for different values of $\epsilon\in [0.5,2,5, 10]$  and $d\in [2,10,15]$. To calculate $S_{\epsilon}(P_n,Q_n)$ and the empirical potentials---which allows us to compute $CI^{n,n}_{0.05}$---we  use the python library  POT, see \cite{Pot}. 
\begin{table}[h!]
\begin{center}
\begin{tabular}{c|c|c|c|c|c|}
\cline{3-6}
\multicolumn{2}{c|}{}&\multicolumn{4}{|c |}{$ \mathbb{P}\left(S_\epsilon(P,Q)\in CI^{n,n}_{0.05}\right)$
}\\
\cline{2-6}
 & $n$& 
$\epsilon=0.5$& $\epsilon=2$&$\epsilon=5$&$\epsilon=10$
\\
\hline
\multicolumn{1}{|c|}{}&50 & 0.935 & 0.936 &  0.932 &  0.941\\
 \multicolumn{1}{|c|}{}&100 &0.937 &  0.952 &  0.929 & 0.941\\
\multicolumn{1}{|c|}{} &250 & 0.95 &  0.94 &  0.935 &  0.949\\
\multicolumn{1}{|c|}{$d=2$}&500 & 0.954 &  0.947 & 0.95 &  0.958\\
\multicolumn{1}{|c|}{} &1000 & 0.944 &  0.954 &  0.947 &  0.96 \\
\multicolumn{1}{|c|}{} & 5000 & 0.939 &  0.957 &  0.947 &  0.955 \\
\hline
\multicolumn{1}{|c|}{}&50 & 0.781 & 0.945 &  0.958 &  0.932\\
 \multicolumn{1}{|c|}{}&100 &0.787 &  0.937 &  0.951 & 0.945\\
\multicolumn{1}{|c|}{} &250 & 0.775 &  0.941  &  0.948 &   0.943\\
\multicolumn{1}{|c|}{$d=10$}&500 & 0.785 &  0.955 & 0.953 &  0.947\\
\multicolumn{1}{|c|}{} &1000 & 0.803 &  0.94 &  0.945 &  0.954 \\
\multicolumn{1}{|c|}{} & 5000 & 0.862 &  0.944 &  0.946 &  0.951 \\
\hline
\multicolumn{1}{|c|}{}&50 & 0.487 & 0.944 &  0.933 &  0.944\\
 \multicolumn{1}{|c|}{}&100 &0.396 & 0.944 &   0.957 & 0.944\\
\multicolumn{1}{|c|}{} &250 &0.271 &  0.938  & 0.943 &    0.953\\
\multicolumn{1}{|c|}{$d=15$}&500 & 0.194 & 0.94 &0.941  & 0.947\\
\multicolumn{1}{|c|}{} &1000 &0.173 &  0.938 &  0.945 &  0.955 \\
\multicolumn{1}{|c|}{} & 5000 & 0.134 &   0.942 & 0.943 &  0.943 \\
\hline
\end{tabular}
\end{center}
\caption{Evolution of the Monte Carlo estimation (number of iterations equals $1000$) of  $\mathbb{P}\left(S_\epsilon(P,Q)\in CI^{n,n}_{0.05}\right)$ for different values of the dimension $d$ and regularization factor $\epsilon$. }
\label{Table}
\end{table}
We observe that, effectively, both  $d$ and $\epsilon$ affect the estimation of the asymptotic confidence interval $ CI^{n,n}_{0.05}$. Actually a large sample size is required to achieve the nominal confidence interval for small values of $\epsilon$ and large dimension. This is more or less expected, in view of Remark~\ref{fast_rate}, the value $n\,\big|\E  S_{\epsilon}(P_n,Q)-S_{\epsilon}(P,Q)\big|$ can be upper bounded by a constant $ {C_\Omega}$, which depends exponentially on the support's diameter--- extrapolating this argument to the case where the probabilities are not supported in a compact set --- and it provides a possible explanation of the inaccuracy produced by the choice of small values  $\epsilon$ or large $d$. Note that this exponential dependency on the diameter is translated directly to an exponential dependence on $\frac{1}{\epsilon}$ by the change of variables $x\mapsto \varepsilon^{-\frac 1 2}x$. Moreover, the convergence, when $\epsilon\to 0$, of the entropic regularised  potentials towards the optimal transport ones (see \cite{NutWie22}), which are cursed by the dimension (see \cite{weedBach}), explains also the results of Table~\ref{Table}.

\section*{Acknowledgements} $ $ \\
\noindent The research of Eustasio del Barrio is partially supported by FEDER,
Spanish Ministerio de Economía y Competitividad, grant MTM2017-86061-C2-1-P and Junta de
Castilla y León, grants VA005P17 and VA002G18. The research of Alberto González-Sanz and Jean-
Michel Loubes is partially supported by the AI Interdisciplinary Institute ANITI, which is funded by
the French “Investing for the Future – PIA3” program under the Grant agreement ANR-19-PI3A-0004.  The research of  Jonathan Niles-Weed is partially funded by National Science Foundation, Grant DMS-2015291.


\end{document}